\documentclass[12pt, a4paper]{amsart}
\usepackage[margin=1.2in]{geometry}
\usepackage{tikz}
\usetikzlibrary{arrows}
\usetikzlibrary{matrix,positioning}
\usepackage{amsmath}
\usepackage{amsthm}
\usepackage{color}
\usepackage{graphicx}
\usepackage{float}
\usepackage{young}
\usepackage[vcentermath]{youngtab}
\usepackage{amsfonts}
\usepackage{hhline}
\usepackage{amssymb}
\usepackage{amscd}
\input xy
\xyoption{all}
\usepackage{booktabs}
\usepackage{array}
\usepackage{multirow}
\numberwithin{equation}{section}
\newtheorem{theorem}{Theorem}[section]

\newtheorem{lemma}[theorem]{Lemma}
\newtheorem{prop}[theorem]{Proposition}

\newtheorem{defi}{Definition}[section]

\theoremstyle{definition}
\newtheorem{example}{Example}[section]
\newtheorem{rem}{Remark}[section]

\newcommand\half{\frac{1}{2}}

\newcommand\be{\beta}

\newcommand\g{\gamma}

\newcommand\D{\Delta}

\newcommand\Dp{\Delta^+}

\renewcommand\a{\alpha}

\newcommand\ganz{\mathbb Z}

\newcommand\s{\sigma}

\newcommand\e{\epsilon}

\newcommand\R{\mathbb R}

\newcommand\G{\Gamma}

\renewcommand{\G}{\Gamma}

\begin{document}
\title[Minimal inversion complete sets]{The maximum cardinality of minimal inversion complete sets  in finite reflection groups}
\author[Malvenuto, M\"oseneder Frajria, Orsina, Papi]{Claudia Malvenuto\\ Pierluigi M\"oseneder Frajria\\Luigi Orsina\\Paolo Papi}

\begin{abstract} We compute for reflection groups of type $A,B,D,F_4,H_3$ and for dihedral groups a statistic counting the maximal cardinality of a set of elements in the group whose generalized inversions yield the full set of inversions and which are minimal with respect to this property. We  also provide  lower bounds
for the $E$  types  that we conjecture to be the exact value of our statistic. 
\end{abstract}
\keywords{Inversion; reflection group; root system; extremal combinatorics of permutations.}
\subjclass[2010]{Primary 	17B22; Secondary 05E10}
\maketitle

\section{Introduction}
Let $S_n$ be the symmetric group on $\{1,\ldots,n\}$. Recall that if $\s\in S_n$,  $1\leq i<j\leq n$ and $\s(i)>\s(j)$,  the pair $(\s(i),\s(j))$  is said to be an inversion for  $\s\in S_n$.  Let $N(\s)$ denote the inversion set of $\s$.  Also set $D=\{(i,j)\mid \,1\leq j<i\leq n\}$. 
\begin{defi}\label{d} We say that $Y\subset S_n$ is {\sl inversion complete} if $\bigcup\limits_{\s\in Y}N(\s)=D$, and that is {\sl minimal inversion complete} if it is inversion complete and minimal  with respect to this property.
\end{defi} 
The following result in extremal combinatorics of permutations  has been communicated to us by Fabio Tardella \cite{T},  who informed us about a forthcoming joint work with M. Queyranne and  E. Balandraud.
\begin{theorem} \label{f} The maximal cardinality of a minimal inversion complete subset of $S_n$ is $\lfloor \frac{n^2}{4}\rfloor$.
\end{theorem}
 The enumerative problem dealt with in Theorem \ref{f} admits a natural generalization to finite reflection groups. Indeed, let $\D$ be a finite 
crystallographic irreducible root system and $W$ be the corresponding  Weyl group. Fix a set of positive roots $\Dp\subset\D$. If $w\in W$, then permutation inversions  are naturally replaced by the subset  $N(w)$ of $\Dp$ defined in \eqref{nw} (cf. \cite{BB}). The problem consists in determining the maximal cardinality of a subset $Y$ of $W$ such that $\bigcup\limits_{w\in Y}N(w)=\Dp$ and minimal with respect to this property. 
Let us denote this number by $MC(T)$ for a group of type $T$.\par In this perspective we provide a proof of Theorem \ref{f} and of the following results. 

\begin{theorem} \label{fb}   $MC(B_n)=\binom{n}{2}+1$.\end{theorem}
\begin{theorem} \label{fd}  $MC(D_n)=\binom{n}{2}$.
\end{theorem}
We also have proofs that 
\begin{equation}\label{geq}MC(F_4)\ge 6,\quad MC(E_6)\geq 16,\quad MC(E_7)\geq 27,\quad MC(E_8)\geq 36.\end{equation}
We conjecture that these bounds  are actually the exact values of our statistic.
We have indeed a computer assisted proof of equality in type $F_4$ (cf. Remark  \ref{f4}).\par
By using the canonical root system (cf. Section \ref{nc}), our problem further generalizes to noncrystallographic types. It is quite easy to prove that $MC(I_2(m))=2$; in particular, since $G_2$ is $I_2(6)$, we have $MC(G_2)=2$. We can  also show that $MC(H_3)=5$  and that $MC(H_4)\geq 8$. (see Section \ref{nc}).

\vskip10pt
Theorems \ref{f}, \ref{fb}, \ref{fd} are proven in two steps. First we exhibit a minimal inversion complete set of the desired cardinality, which gives a lower bound for $MC(T)$. The choice of this set is motivated by the theory of abelian ideals of Borel subalgebras. Then we prove that the desired cardinality is the maximal allowed for any minimal inversion complete set: 
this reduces to graph theoretical  considerations, which arise from Lie theoretical arguments. This step is a direct consequence of a well known result in extremal graph theory in type $A$, and it is more involved in the other cases.\par
In Section \ref{exce} we prove the lower bounds \eqref{geq}. Let $\mathfrak g$ be the simple Lie algebra  having $\D$ as root system, and  let $\mathfrak b$ be a Borel subalgebra of 
 $\mathfrak g$. As in the classical cases, these bounds  are   dimensions of abelian ideals of $\mathfrak b$.
In Section \ref{nc} we deal with noncrystallographic types. In Section \ref{spec} we make a speculation relating  the calculation of our statistics with an arithmetic property of subsets of positive roots.  Were our speculation    true, we would obtain a clear  explanation of the relationship between 
$MC(T)$ and the abelian ideals of $\mathfrak b$.
\section{Setup}
Let $\D$ be an irreducible crystallographic  root system in a Euclidean space $V$, endowed with a positive definite symmetric bilinear form 
$(\cdot,\cdot)$. We call a set $\Sigma$
  of roots symmetric if $\Sigma=-\Sigma$. 
  
 Fix a positive system $\Dp$ and let  $\Pi=\{\a_1,\ldots,\a_n\}$ be the corresponding basis of simple roots. Recall that the support of a root $\a$ is the subgraph of the Dynkin diagram whose vertices are the simple roots in which $\a$ has nonzero coefficient (cf. \cite[3.3, Definition 1]{B}).\par
  
   Let $W$ be the corresponding Weyl group, i.e. the group generated by the reflections $s_\a$ for $\a\in \D$. 
  Set $s_i=s_{\a_i},\,i=1,\ldots,n$. Recall that $\{s_i\mid 1 \leq i\leq n\}$ is a set of Coxeter generators for $W$. Let $\ell$  be the corresponding length function. \par

Set  
\begin{equation}\label{nw}N(w)=\{\a\in\Dp\mid w^{-1}(\a)\in -\Dp\}.
\end{equation}
Note  that 
\begin{equation}\label{rinw}
N(w)=\{\a\in\Dp\mid \ell(ws _\a)<\ell(w)\}.\end{equation}
Then the map $w\mapsto N(w)$ is injective. The set $N(w)$ can be recovered as follows:
if $w=s_{i_1}\cdots s_{i_k}$ is a reduced expression of $w$, then 
\begin{equation}\label{nww}
N(w)=\{\a_{i_1},s_{i_1}(\a_{i_2}),\ldots,s_{i_1}\cdots s_{i_{k-1}}(\a_{i_k})\}.
\end{equation}
Recall (cf. \cite[Chapter 3]{BB}) that the right weak Bruhat order on $W$ is the partial order defined by $v\leq w$ if there exists $u\in W$ such that 
$w=vu$ and $\ell(w)=\ell(v)+\ell(u)$. It is well-known (cf. \cite[Proposition 3.1.3]{BB}) that 
\begin{equation}\label{weak}v\leq w \iff N(v)\subseteq N(w).\end{equation}
The following characterization of the sets $N(w)$ is given e.g. in \cite[VI, Exercise 16]{B}. Let us say that $L\subseteq \Dp$ is {\it closed} if 
$$\a,\beta\in L,\,\a+\beta\in\Dp\implies \a+\beta\in L,
$$ 
 {\it coclosed} if 
 $$\a,\beta\notin  L,\,\a+\beta\in\Dp\implies \a+\beta\notin L,
 $$
 and {\it biclosed} if it is both closed  and coclosed. Here we follow \cite{Dy} for terminology.
 
\begin{prop} Given $L\subseteq \Dp$, then $L=N(w)$ for some $w\in W$ if and only if it is biclosed.
\end{prop}
Notice that coclosedness can be expressed as follows:
\begin{equation}\label{p2}
\a+\beta\in L ,\a,\be\in\Dp, \a\notin L\implies \be\in L.
\end{equation}
 \par
 \vskip5pt
\noindent

\section{Minimal inversion complete sets}\label{3} We start by recalling our  natural generalization to Weyl groups of the notion of minimal inversion complete sets given in Definition \ref{d}  for subsets of the symmetric group.

\begin{defi}\label{mincomp} We say that $Y\subset W$ is {\sl inversion complete} if $\bigcup\limits_{x\in Y}N(x)=\Dp$, and that is {\sl minimal inversion complete} if it is inversion complete and minimal  with respect to this property.
\end{defi} 
\begin{rem}ÊStrictly speaking, this definition depends on $\D$. The ambiguity might occur for the hyperoctahedral group, which is the Weyl group of root systems of type $B$ and  of type $C$ as well. 

If $\a\in V$, we set $\a^\vee=\frac{2}{(\a,\a)}\a$. Let $\D_B$ be a root system of type $B$ and set $\D_C=\D_B^\vee$. Then $\D_C$ is a root system of type $C$. 
Clearly  the same Weyl group $W$ corresponds to both root systems, but, if $w\in W$, the set $N(w)$ defined in \eqref{nw} does depend on the choice for $\D$. Write $N_B(w)$ and $N_C(w)$ to highlight the dependence on  the root system.
Since $w(\a^\vee)=w(\a)^\vee$, it is clear that $N_C(w)=N_B(w)^\vee$, so  
$$\bigcup\limits_{x\in Y}N_B(x)=\Dp_B \iff \bigcup\limits_{x\in Y}N_C(x)=\Dp_C.$$
Therefore, being inversion complete depends only on $W$.
\end{rem}
\begin{rem}\label{antichain}ÊIt is immediate from Definition \ref{mincomp} and relation \eqref{weak} that a minimal inversion complete set is an anti-chain in the weak Bruhat order.
\end{rem}
\vskip5pt
Theorem \ref{f} leads to the following\vskip5pt\noindent
{\bf Problem.}  Find the maximal cardinality $MC(T)$ of a minimal inversion complete set in a finite Weyl group of type $T$.Ê

 \begin{rem} In Section \ref{nc} Êwe will formulate the above problem in the more general context of finite reflection groups.
\end{rem}

\begin{example}\label{G2} If $W$ is a dihedral Weyl group of order $2m$, then   $MC(W)=2$.
 Indeed, if $\{s_1,s_2\}$ is a set of Coxeter generators for $W$, then the unique reduced expressions are obtained by taking prefixes or suffixes of the word $\underbrace{s_1s_2\cdots s_1s_2}_{m}$. In particular, any anti-chain has cardinality at most $2$, and by Remark \ref{antichain} we have that 
 $MC(W)\leq 2$. The converse inequality is obvious: for any $k$, $1\leq k\leq m-1$, the subset $\{\underbrace{s_1s_2\cdots }_{m-k},\underbrace{s_2s_1\cdots}_{k}\}$ is minimal complete. Since the Weyl groups of type $B_2, G_2$ coincide with $I_2(4),\,I_2(6)$, respectively, this argument proves in particular that $MC(G_2)=MC(B_2)=2$.\end{example}
We now introduce a very useful tool for our analysis.
\begin{defi}
If $Y$ is a subset of $W$, an essential root for $Y$ is a root $\a\in \Dp$ such that there is a unique $w\in Y$ with $\a\in N(w)$. In such a case,  we will write $N(\a)$ instead of $N(w)$. 

We say that $Y$ admits an essential set of roots, if, for each $w\in W$, there is an essential root $\a$ with $\a\in N(w)$. 

If  $Y$ admits an essential set of roots, then we can  choose one essential root $\a(w)\in N(w)$ for each $w\in W$. The set $\{\a(w)\mid w\in Y\}$ is called an essential set for $Y$.

We call a subset $E$ of $\Dp$ an essential set, if there is a subset $Y$ of $W$ such that $E$ is an essential set for $Y$.
\end{defi}

\begin{rem} If  $Y$ is minimal inversion complete, then there is at least one essential root in $N(w)\subset \Dp$ for each $w\in Y$, hence $Y$ admits an essential set of roots.
\end{rem}

We call a $k$-tuple $P=(\eta_1,\dots , \eta_k)\in (\D^+)^k$ a root path if   $\sum_{i=1}^j\eta_i\in\D^+$ for any $j=1,\dots,k$. We call the root $\a(P)=\sum_{i=1}^k\eta_i$ the last sum of the root path. If  $P=(\eta_1,\dots , \eta_k)$ is a root path, we let $Supp(P)=\{\eta_1,\dots , \eta_k\}$. If $S$ is a subset of $\Dp$ and $Supp(P)\subset S$, then we will say somewhat loosely that $P$ is a path in $S$.

\begin{lemma}\label{paths}Fix an essential set $S^+$ for $Y$ and let $P=(\eta_1,\dots , \eta_k)$ be a root path in $S^+$.
If  $\a(P)\in N(\gamma)$ for $\gamma\in S^+$, then $\gamma\in Supp(P)$.
\end{lemma}
\begin{proof}
The proof is by induction on $k$. If $k=1$ then, since $\gamma,\eta_1\in S^+$, we see that $\eta_1\in N(\gamma)$ if and only if $\eta_1=\gamma$. Assume now $k>1$. If $\a(P)\in N(\gamma)$ and $\gamma\not\in Supp(P)$ then, in particular, $\gamma\not\in \{\eta_1,\dots , \eta_{k-1}\}$, so, by the induction hypothesis, $\sum_{i=1}^{k-1}\eta_i\not\in N(\gamma)$. Since $\eta_k\ne \gamma$, then $\eta_k\not\in N(\gamma)$, hence  $\a(P)=(\sum_{i=1}^{k-1}\eta_i)+\eta_k\not\in N(\gamma)$, a contradiction.
\end{proof}

 Recall from e.g. \cite[8.5]{Hum} that, given $\a,\be\in\D$, the $\a$-string through $\be$ is the set of roots of the form $\be+ i\a$ with $i\in \ganz$. We will refer to such sets of roots as {\it root strings}. It is well known that there are integers $r\le q$ such that the $\a$-string through $\be$ is $\{\be+i\a\mid r\le i \le q\}$.  We call $|q-r|$
 the length of the root string.
 
\begin{prop}\label{conditions} Let $S^+\subset \Dp$ be an essential set for $Y$  and set $S=S^+\cup (-S^+)$. Then: 
\begin{enumerate}
\item\label{nr1} If there are root paths $P,P'$ in $S^+$ such that $\a(P)=\a(P')$ then $Supp(P)\cap Supp( P')\ne\emptyset$.
\item\label{nr2} if $\a,\be\in S$ and $\a+\be\in \D$ then $\a+\be\not\in S$,
\item \label{nr3}$S$ does not contain root strings of length greater or equal to $2$.
\end{enumerate}
\end{prop}
\begin{proof}
Let us prove \eqref{nr1}. Since $Y$ is minimal inversion complete there is $\eta\in S^+$ such that $\a(P)=\a(P')\in N(\eta)$. Lemma \ref{paths} now implies that $\eta\in Supp(P)\cap Supp( P')$. 

Let us prove \eqref{nr2}. If $\a,\be\in S^+$ and $\a+\be\in S^+$, then $P=(\a,\be)$ and $P'=(\a+\be)$ are root paths that contradict \eqref{nr1}, thus  $\a+\be\not\in S^+$. Clearly this argument rules out also the case when $\a,\be\in -S^+$. If $\a\in S^+$ and $\be\in-S^+$ and $\a+\be\in S^+$ (or $\a+\be\in -S^+$), then $\a=-\be+(\a+\be)\in S^+$ (or $\be=-\a+(\a+\be)\in -S^+$), so  the previous argument applies.
\par
Let us prove \eqref{nr3}. If $S$ contains a root string of length greater or equal to $2$ then there are $\a\in S$ and $\g\in \D$ such that $\{\a,\a+\g,\a+2\g\}\subset S$. We can assume $\gamma\in\Dp$: if $\gamma\in-\Dp$ then it is enough to show that the string $-\a,-\a-\g,-\a-2\g$ cannot stay in $S$.
If $\a\in S^+$ then $\a\not\in N(\a+\g)$ so $\g\in N(\a+\g)$ by coclosedness \eqref{p2} of $ N(\a+\g)$. But this implies $\a+2\g\in N(\a+\g)$ by closedness of  $N(\a+\g)$, since $\a+\g\in N(\a+\g)$, a contradiction. If $\a\in - S^+$ and $\eta=\a+\g\in S^+$ then $\g=-\a+\eta$, so $\a+2\g=-\a+2\eta$, but then $\a+2\g$ is the last sum of the root path $(-\a,\eta,\eta)$ and this is not possible by \eqref{nr1}. The other cases are treated by showing that the string $-\a-2\g,-\a-\g,-\a$ cannot stay in $S$.
\end{proof}

\section{Type $A$} In this section we will give a proof of Theorem \ref{f}, exploiting the fact that the symmetric group is the Weyl group of a crystallographic root system $\D$ of type $A$. 
\subsection{Encoding the roots}We start by recalling the explicit description of the root system of type $A_{n-1}$ as given in Planche I of \cite{B}.
Let $\{\e_i\}_{i=1}^n$ denotes the standard basis of $\R^n$, then $\D=\{\e_i-\e_j\mid1 \leq i\ne j\leq n\}$ is a root system of type $A_{n-1}$ and we may choose 
 $\Dp=\{\e_i-\e_j\mid1 \leq i<j\leq n\}$.

We associate to a set $S$ of roots a directed graph $G(S)=(V(S),E(S))$ as follows: $$V(S)=\{1,\dots,n\}\text{ and }
(i,j)\in E(S) \Leftrightarrow \e_i-\e_j\in S.
$$
 Note that $\a,\be,\a+\be\in S$ if and only if they form in the graph $G(S)$ a transitive triangle, with $\a$ and $\be$ represented by consecutive arcs:
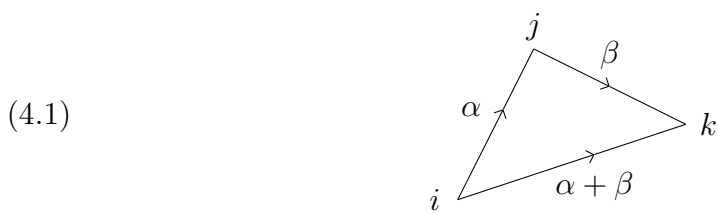
\begin{equation}
\begin{tikzpicture}[>=angle 90,baseline=0]
\draw[->] (0,-1)--(0.6,0.2);
\draw(0.6,0.2)--(1,1);
 \draw[->](1,1)-- (2,0.5);
  \draw[-](2,0.5)-- (3,0);
  \draw[->](0,-1)-- (1.8,-0.4);
  \draw[-](1.8,-0.4)-- (3,0);
   \node at (-0.3,-1) {$i$};
   \node at (1,1.3) {$j$};
   \node at (3.3,0) {$k$};
    \node[xshift=-12pt,yshift=0pt] at (0.6,0.2) {$\a$};
    \node[xshift=0pt,yshift=-12pt] at (1.8,-0.4) {$\a+\be$};
     \node[xshift=0pt,yshift=12pt] at (2,0.5){$\be$};
\end{tikzpicture}
\end{equation}

The graph of a symmetric set $S$ has only antiparallel arcs so the graph of $S$ can be substituted by the underlying undirected graph that we denote by $\G(S)$.

\begin{rem}\label{no triangle}
Let $S$ be a symmetric set of roots. Then $S$ satisfies condition \eqref{nr2} of Proposition \ref{conditions} if and only if there are no triangles in $\G(S)$. Indeed, if $\a,\be\in S$  and $\a+\be$ is a root, then  $\a+\be\not\in S$ for, otherwise, there would be a triangle. On the other hand, if there is a triangle in $\G(S)$, then, since all arcs in $G(S)$ are antiparallel, there is a transitive triangle in $G(S)$, hence condition \eqref{nr2} is not satisfied. 
\end{rem}

 Recall that $\Dp$ is a poset (usually named the {\it root poset}) under the partial order $\preceq$ defined by $\a\preceq \be$ if and only if $\be-\a$ is a sum of positive roots or zero
 (cf. e.g. \cite[10.1]{Hum}). 
 The arcs of its Hasse diagram can be labelled by the simple roots: an arc labelled by $\a_i$ joins $\a$ to $\a+\a_i$. The Hasse diagram can be represented by a staircase diagram:

$$
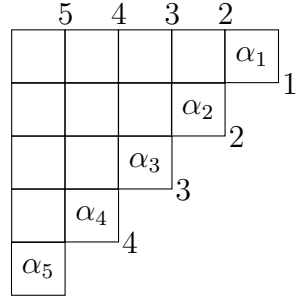
\begin{tikzpicture}[baseline=0,cell/.style={rectangle,draw,minimum size=7mm},rowlab/.style={xshift=4pt},collab/.style={yshift=6pt}]
\matrix[row sep=-0.1mm,column sep=-0.5pt] {\node[cell](11){};&\node[cell](12){};&\node[cell](13){};&\node[cell](14){};&\node[cell](15){$\a_1$};\\
\node[cell](21){};&\node[cell](22){};&\node[cell](23){};&\node[cell](24){$\a_2$};&\\
\node[cell](31){};&\node[cell](32){};&\node[cell](33){$\a_3$};&&\\
\node[cell](41){};&\node[cell](42){$\a_4$};&&&\\
\node[cell](51){$\a_5$};&&&&\\};
\node[rowlab] at (15.south east) {$1$};
\node[rowlab] at (24.south east) {$2$};
\node[rowlab] at (33.south east) {$3$};
\node[rowlab] at (42.south east) {$4$};
\node[collab] at (11.north east) {$5$};
\node[collab] at (12.north east) {$4$};
\node[collab] at (13.north east) {$3$};
\node[collab] at (14.north east) {$2$};
\end{tikzpicture}
$$

 Here the diagram for type $A_5$ is shown. The cells are the vertices of the Hasse diagram and the border between two cells is the arc between the two vertices. The direction of the arc is right-left for a vertical border and down-up for a horizontal border. The label for each arc is given by the number at its right for a horizontal border and the number on top of it for a vertical border. This means that, if the  border belongs to a line labelled by $i$, then the corresponding arc in the Hasse diagram is decorated by $\a_i$. 
 In particular, an element  of the poset corresponds to the root obtained by choosing a descending path (i.e. going right and down) from the corresponding cell to a simple root $\a_i$ and then  summing to $\a_i$ the roots labeling the arcs of the descending path.  For example, the highest root  $\theta$ (i.e. the cell on the upper left corner) can be reached following the first row hence $\theta=\a_1+\a_2+\a_3+\a_4+\a_5$. In the following we identify positive roots and cells using the above encoding.
 
 Passing from this encoding to the $\e$-coordinates is very easy: label the $i$-th row from the top of the diagram with $i$, the $j$-th column from the right with $j+1$. Then the root $\e_i-\e_j$ corresponds to the $(i,j)$-cell. For example, consider the gray cells in the following figure.
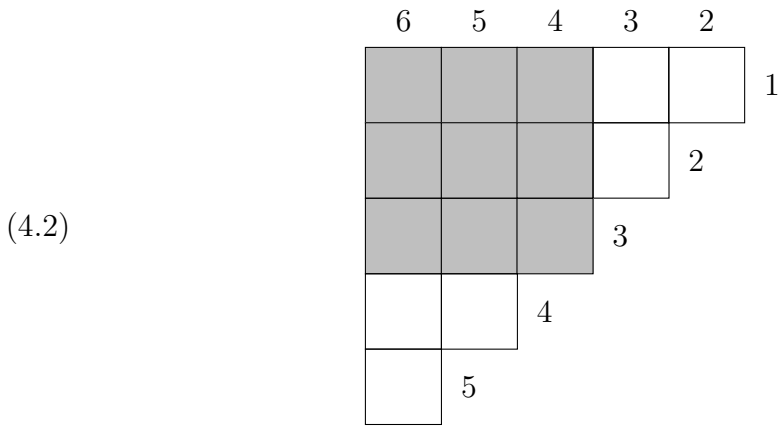
\begin{equation}\label{A5}
\begin{tikzpicture}[baseline=0,cell/.style={rectangle,draw,minimum size=10mm},phcell/.style={minimum size=10mm},big/.style={rectangle,draw,minimum size=10mm,minimum width=15mm},rowlab/.style={xshift=10pt},collab/.style={yshift=10pt}]
\matrix[row sep=-0.4pt,column sep=-0.5pt] {
\node[cell,fill=gray!50](11){};&\node[cell,fill=gray!50](12){};&\node[cell,fill=gray!50](13){};&\node[cell](14){};&\node[cell](15){};\\
\node[cell,fill=gray!50]{};&\node[cell,fill=gray!50]{};&\node[cell,fill=gray!50]{};&\node[cell](24){};\\
\node[cell,fill=gray!50]{};&\node[cell,fill=gray!50]{};&\node[cell,fill=gray!50](33){};\\
\node[cell]{};&\node[cell](42){};\\
\node[cell](51){};\\};
\node[rowlab] at (15.east) {$1$};
\node[rowlab] at (24.east) {$2$};
\node[rowlab] at (33.east) {$3$};
\node[rowlab] at (42.east) {$4$};
\node[rowlab] at (51.east) {$5$};
\node[collab] at (15.north) {$2$};
\node[collab] at (14.north) {$3$};
\node[collab] at (13.north) {$4$};
\node[collab] at (12.north) {$5$};
\node[collab] at (11.north) {$6$};
\end{tikzpicture}
\end{equation}
The corresponding set of roots is
$$
\{\e_i-\e_j\mid 3\le i<j\le 6\}.
$$

 \subsection{Proof of Theorem \ref{f}}
   After the identification $\e_i-\e_j\leftrightarrow (j,i)$, \eqref{nw} yields exactly for $N(w)$ the set of the inversions of the permutation $w\in S_n$, so the proof of Theorem \ref{f} reduces to showing  that $MC(A_{n-1})=\lfloor \frac{n^2}{4}\rfloor$.

Set   $N=\lfloor \frac{n}{2}\rfloor,\,N'=\lceil \frac{n}{2}\rceil$. Define
\begin{equation}\label{A}P(A_{n-1})=\{\e_i-\e_{n-j+1}\in\Dp\mid 1\leq i\leq N, 1\leq j\leq N'\}.\end{equation} 
The set $P(A_5)$ is the highlighted subset in the diagram \eqref{A5}, while $P(A_6)$  is the highlighted subset of the diagram below.
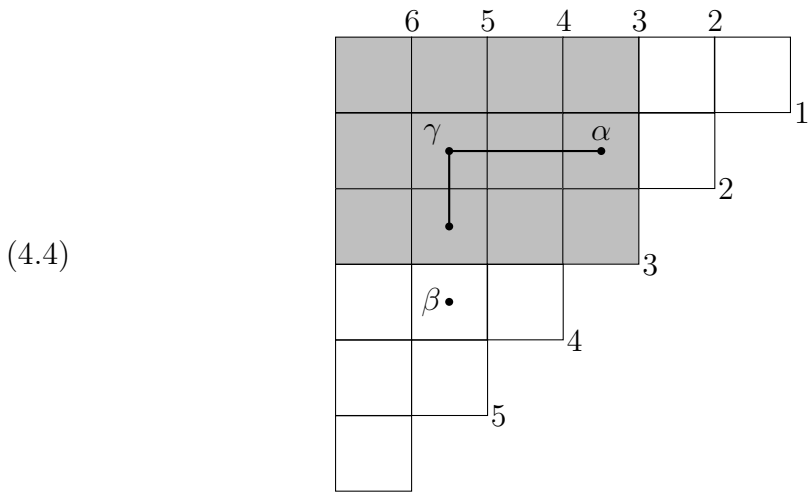
\begin{equation}\label{A6}
\begin{tikzpicture}[baseline=0,cell/.style={rectangle,draw,minimum size=10mm},rowlab/.style={xshift=4pt},collab/.style={yshift=6pt},dot/.style={circle,fill,inner sep=0,minimum size=3pt}]
\matrix[row sep=-0.1mm,column sep=-0.5pt] {\node[cell,fill=gray!50](11){};&\node[cell,fill=gray!50](12){};&\node[cell,fill=gray!50](13){};&\node[cell,fill=gray!50](14){};&\node[cell](15){};&\node[cell](16){};\\
\node[cell,fill=gray!50](21){};&\node[cell,fill=gray!50](22){};&\node[cell,fill=gray!50](23){};&\node[cell,fill=gray!50](24){};&\node[cell](25){};&\\
\node[cell,fill=gray!50](31){};&\node[cell,fill=gray!50](32){};&\node[cell,fill=gray!50](33){};&\node[cell,fill=gray!50](34){};&&\\
\node[cell](41){};&\node[cell](42){};&\node[cell](43){};&&&\\
\node[cell](51){};&\node[cell](52){};&&&&\\
\node[cell](51){};&&&&&\\};
\draw[thick](-1.5,0.5)--(-1.5,1.5)--(0.5,1.5);
\node[rowlab] at (16.south east) {$1$};
\node[rowlab] at (25.south east) {$2$};
\node[rowlab] at (34.south east) {$3$};
\node[rowlab] at (43.south east) {$4$};
\node[rowlab] at (52.south east) {$5$};
\node[collab] at (11.north east) {$6$};
\node[collab] at (12.north east) {$5$};
\node[collab] at (13.north east) {$4$};
\node[collab] at (14.north east) {$3$};
\node[collab] at (15.north east) {$2$};
\node[dot]  at (-1.5,0.5){};
\node[dot](gamma)  at (-1.5,1.5){};
\node[dot] (alfa) at (0.5,1.5){};
\node[dot] (beta) at (-1.5,-0.5){};
\node[yshift=5pt,xshift=-5pt] at (gamma.north west) {$\gamma$};
\node[xshift=-5pt] at (beta.west) {$\be$};
\node[yshift=5pt]  at (alfa.north) {$\alpha$};
\end{tikzpicture}
\end{equation}

This set has  the following two properties:
\begin{enumerate}
\item For any $\gamma\in P(A_{n-1})$, $\gamma=\a+\beta,\,\a,\beta\in\Dp$, then either $\a\in P(A_{n-1}),\beta\notin P(A_{n-1})$ or $\a\notin P(A_{n-1}),\beta\in P(A_{n-1})$.
\item Any positive root $\a\in \Dp\backslash P(A_{n-1})$ appears as a summand in a decomposition of a root of $P(A_{n-1})$ as a sum of two positive roots.
\end{enumerate}
These properties are easily checked using the following remark. Given a positive root $\gamma$, consider its hook within the above diagram. A decomposition of $\gamma$ as a sum of two positive roots  is given by  the rightmost cell in the arm of the hook and by  the root just below $\gamma$ in the leg. For example, in \eqref{A6}, $\g=\a_2+\a_3+\a_4+\a_5=\a+\be$ with $\a=\a_2+\a_3$, $\be=\a_4+\a_5$. The other decompositions (if any) are obtained by moving by one step left in the arm and down in the leg, at the same time. 
For any $\gamma\in P(A_{n-1})$, set
$$L(\gamma)=\{ \gamma\}\cup\{\a\in\Dp\mid \a\notin P(A_{n-1}), \gamma-\a\in \Dp\}.$$
\begin{lemma} 
\label{31} For any $\gamma\in P(A_{n-1})$, the set $L(\gamma)$ is biclosed.
\end{lemma}
\begin{proof} To prove closedness,  we observe that no two roots from $L(\gamma)$ sum to a root. This is clear if the two summands both  belong 
to the arm (or to the leg) of $\gamma$, and it follows from the fact that the union of their supports is disconnected 
in the other cases.
\par To check co-closedness  we use \eqref{p2}. 
Let $\a,\beta$ be positive roots such that $\a+\be\in L(\gamma)\setminus\{\gamma\}$, and assume that $\a\notin L(\gamma)$. We now prove that $\be\in L(\gamma)$. This is clear when $\a+\be=\gamma$ because of property (1) of $P(A_{n-1})$ above. If $\a+\be\in L(\gamma)\backslash\{\gamma\}$, then  $\gamma=\a+\beta+\eta$
 for some $\eta\in P(A_{n-1})$. Then, by the definition of $P(A_{n-1})$,  both $\a$ and $\be$ do not belong to $P(A_{n-1})$ and by construction there is exactly one of them, necessarily  $\be$, such that 
 $\gamma-\be\in\Dp$. Then $\be\in L(\gamma)$ and \eqref{p2} is fulfilled.
 \end{proof}

Let $w_\gamma\in S_n$ be such that $N(w_\gamma)=L(\gamma)$ and set $Y(A_{n-1})=\{w_\gamma\mid \gamma\in P(A_{n-1})\}$. 
\begin{prop}\label{32} $Y(A_{n-1})$ is a minimal inversion complete set. In particular, we have that  $MC(A_{n-1})\geq\lfloor \frac{n^2}{4}\rfloor$.
\end{prop}
\begin{proof} The fact that $\bigcup\limits_{w\in S_n}N(w)=\Dp$ follows from property (2). Minimality is clear by construction: any root  $\gamma\in P(A_{n-1})$ appears only in $N(w_\gamma)$.
\end{proof}

\begin{proof}[Proof of  Theorem \ref{f}] By  Proposition \ref{32}, it suffices to prove that a minimal inversion complete set $Y$ has at most   $\lfloor \frac{n^2}{4}\rfloor$ elements. Let $S^+$ be an essential set for $Y$ and set $S=S^+\cup -S^+$.  By Proposition \ref{conditions} \eqref{nr3} and Remark  \ref{no triangle}, there is no triangle in $\G(S)$. By  Mantel's Theorem (see \cite{Tu}) a triangle-avoiding graph on $n$ vertices has at most  $\lfloor \frac{n^2}{4}\rfloor$ edges.
\end{proof} 
\section{Type $B$}\label{Type B}
\subsection{Encoding the roots} 
Recall (cf. Planche II of \cite{B}) that the roots for type $B_n$ are, in $\epsilon$-coordinates,
$$
\D=\{\pm(\e_i\pm\e_j)\mid 1\le i< j\le n\}\cup\{\pm\e_i\mid 1\le i\le n\}.
$$
A set of positive roots is
$$
\Dp=\{\e_i\pm\e_j\mid 1\le i< j\le n\}\cup\{\e_i\mid 1\le i\le n\}.
$$

We associate to a set $S$ of roots a directed graph $G(S)=(V(S),E(S))$ as follows: $$V(S)=\{0,\pm1,\dots,\pm n\}\text{ and }
(i,j)\in E(S) \Leftrightarrow sgn(i)\e_{|i|}-sgn(j)\e_{|j|}\in S.
$$

 Note that a root is represented by two arcs: $i\to j$ and $-j\to -i$. Indeed the graphs obtained in this way are invariant under the arc-reversing transformation $i\mapsto -i$. For example the root $\e_1-\e_2$ corresponds to the arcs $1\to 2$ and $-2\to -1$, while the root $\e_1$ corresponds to the arcs $1\to 0$ and $0\to -1$. This annoying feature is balanced by the fact that $\a,\be,\a+\be\in S$ if and only if they form in the graph $G(S)$ a pair of transitive triangles, with $\a$ and $\be$ represented by consecutive arcs:
\begin{equation}\label{sumofroots}
\begin{tikzpicture}[>=angle 90,baseline=0]
\draw[->] (0,-1)--(0.6,0.2);
\draw(0.6,0.2)--(1,1);
 \draw[->](1,1)-- (2,0.5);
  \draw[-](2,0.5)-- (3,0);
  \draw[->](0,-1)-- (1.8,-0.4);
  \draw[-](1.8,-0.4)-- (3,0);
   \node at (-0.3,-1) {$i$};
   \node at (1,1.3) {$j$};
   \node at (3.3,0) {$k$};
       \node[xshift=-12pt,yshift=0pt] at (0.6,0.2) {$\a$};
    \node[xshift=0pt,yshift=-12pt] at (1.8,-0.4) {$\a+\be$};
     \node[xshift=0pt,yshift=12pt] at (2,0.5){$\be$};
\begin{scope}[xshift=5cm,>=angle 90]
\draw[-] (0,-1)--(0.5,0);
\draw[<-](0.5,0)--(1,1);
 \draw[-](1,1)-- (1.8,0.6);
  \draw[<-](1.8,0.6)-- (3,0);
  \draw[-](0,-1)-- (1.5,-0.5);
  \draw[<-](1.5,-0.5)-- (3,0);
   \node at (-0.3,-1) {$-i$};
   \node at (1,1.3) {$-j$};
   \node at (3.3,0) {$-k$};
       \node[xshift=-12pt,yshift=0pt] at (0.6,0.2) {$\a$};
    \node[xshift=0pt,yshift=-12pt] at (1.8,-0.4) {$\a+\be$};
     \node[xshift=0pt,yshift=12pt] at (2,0.5){$\be$};
  \end{scope}
\end{tikzpicture}
\end{equation}

 
We can also associate to $S$ another graph that we call the graph of type $C$ of $S$. It is the graph we would obtain if we used a root system of type $C$. Explicitly the graph of type $C$ is the digraph $G_C(S)=(V_C(S),E_C(S))$, where
 $V_C(S)=\{\pm1,\dots,\pm n\}$ and 
 $$
(i,j)\in E_C(S) \Leftrightarrow sgn(i)\e_{|i|}-sgn(j)\e_{|j|}\in S\text{ and }\half(sgn(i)\e_{|i|}-sgn(j)\e_{|j|})\in S
$$
Note that a long root is represented by two arcs: $i\to j$ and $-j\to -i$, while the short roots are represented by a single arc from $i$ to $-i$.  

We can go from $G(S)$ to $G_C(S)$ by substituting any pair of consecutive arcs $i\to 0\to -i$ in $G(S)$ with a single arc $i\to -i$ in $G_C(S)$.  
For example, if 
$G(S)$ is
$
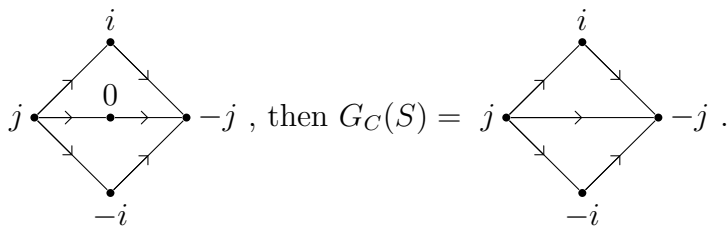
\begin{tikzpicture}[dot/.style={circle,fill,inner sep=0,minimum size=3pt},left/.style={xshift=-5pt},up/.style={yshift=7pt},down/.style={yshift=-7pt},right/.style={xshift=10pt},baseline=-3pt,>=angle 90]
\node[dot] (1) at (-1,0){};
\node[dot] (2) at (0,0){};
\node[dot] (3) at (1,0){};
\node[dot] (4) at (0,1){};
\node[dot] (5) at (0,-1){};
\draw[-] (1) -- (4);
\draw[->] (1) -- (-0.5,0.5);
\draw[-] (4) -- (3);
\draw[->] (4) -- (0.5,0.5);
\draw[-] (3) -- (5);
\draw[->] (5) -- (0.5,-0.5);
\draw[-] (5) -- (1);
\draw[->] (1) -- (-0.5,-0.5);
\draw[-] (1) -- (2);
\draw[->] (1) -- (-0.5,0);
\draw[-] (2) -- (3);
\draw[->] (2) -- (0.5,0);
\node[up] at (4.north){$i$};
\node[up] at (2.north){$0$};
\node[left] at (1.west){$j$};
\node[down] at (5.south){$-i$};
\node[right] at (3.east){$-j$};
\end{tikzpicture}
$, then
$G_C(S)=\begin{tikzpicture}[dot/.style={circle,fill,inner sep=0,minimum size=3pt},left/.style={xshift=-5pt},up/.style={yshift=7pt},down/.style={yshift=-7pt},right/.style={xshift=10pt},baseline=-3pt,>=angle 90]
\node[dot] (1) at (-1,0){};
\node[dot] (3) at (1,0){};
\node[dot] (4) at (0,1){};
\node[dot] (5) at (0,-1){};
\draw[-] (1) -- (4);
\draw[->] (1) -- (-0.5,0.5);
\draw[-] (4) -- (3);
\draw[->] (4) -- (0.5,0.5);
\draw[-] (3) -- (5);
\draw[->] (5) -- (0.5,-0.5);
\draw[-] (5) -- (1);
\draw[->] (1) -- (-0.5,-0.5);
\draw[-] (1) -- (3);
\draw[->] (1) -- (0,0);
\node[up] at (4.north){$i$};
\node[left] at (1.west){$j$};
\node[down] at (5.south){$-i$};
\node[right] at (3.east){$-j$};
\end{tikzpicture}$.

For a symmetric set of roots $S$ both $G(S)$ and $G_C(S)$  have only antiparallel arcs $(i,j)$ and $(j,i)$, so they can be described by the underlying undirected graphs. We denote these undirected graphs by $\G(S)$ and $\G_C(S)$ respectively.
\vskip10pt
 In analogy with type $A$, we represent the Hasse diagram of the ranked poset $\Dp$  by a {\it staircase} diagram:
$$
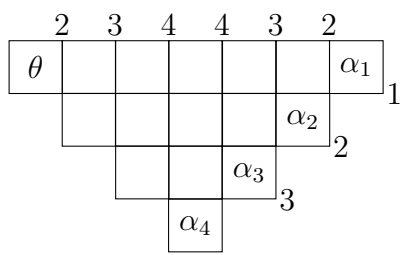
\begin{tikzpicture}[cell/.style={rectangle,draw,minimum size=7mm},rowlab/.style={xshift=4pt},collab/.style={yshift=6pt}]
\matrix[row sep=-0.4pt,column sep=-0.4pt] {\node[cell](11){$\theta$};&\node[cell](12){};&\node[cell](13){};&\node[cell](14){};&\node[cell](15){};&\node[cell](16){};&\node (17) [cell]{$\a_1$};\\
&\node[cell]{};&\node[cell]{};&\node[cell]{};&\node[cell]{};&\node[cell](26){$\a_2$};\\
&&\node[cell]{};&\node[cell]{};&\node[cell](35){$\a_3$};&&\\
&&&\node[cell]{$\a_4$};&&&\\};
\node[rowlab] at (17.south east) {$1$};
\node[rowlab] at (26.south east) {$2$};
\node[rowlab] at (35.south east) {$3$};
\node[collab] at (17.north west) {$2$};
\node[collab] at (16.north west) {$3$};
\node[collab] at (15.north west) {$4$};
\node[collab] at (14.north west) {$4$};
\node[collab] at (13.north west) {$3$};
\node[collab] at (12.north west) {$2$};
\end{tikzpicture}
$$
 Here the diagram for type $B_4$ is shown. The cells are the vertices of the Hasse diagram and the border between two cells is the arc between the two vertices. The direction of the arc is right-left for a vertical border and down-up for a horizontal border. As for type $A$ the labels for the arcs are given by the numbers at the right  and  top of the borders. Similarly, a vertex of the poset corresponds to the root obtained by choosing a descending path (i.e. going right and down) from the corresponding cell to a simple root $\a_i$ and then  summing to $\a_i$ the roots labeling the arcs of the descending path. For example, the highest root  $\theta$ can be reached following the first row hence $\theta=\a_1+\a_2+\a_3+\a_4+\a_4+\a_3+\a_2=\a_1+2\a_2+2\a_3+2\a_4$.

Passing from this encoding to the $\e$-coordinates is very easy: label the $i$-th row from the top of the diagram with $i$, the $j$-th column from the right with $j+1$ if $j<n$, with $0$ if $j=n$, and with $j-2n-1$ if $j>n$. Then the root corresponding to the cell $(i,j)$ is $\e_i-sgn(j)\e_{|j|}$.

For example, consider the following diagram.
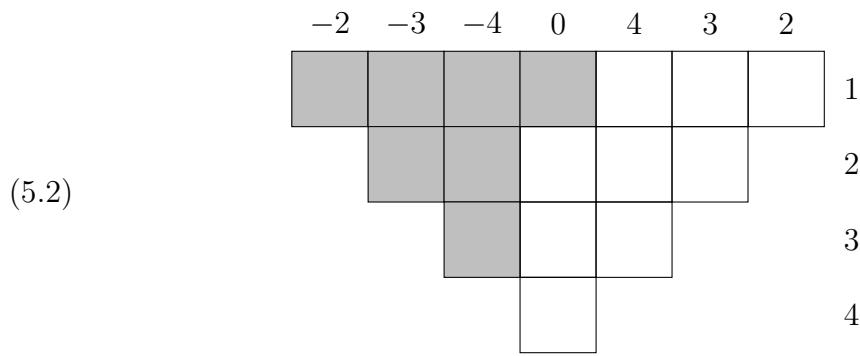
\begin{equation}\label{PBN}
\begin{tikzpicture}[cell/.style={rectangle,draw,minimum size=10mm},phcell/.style={minimum size=10mm},rowlab/.style={xshift=10pt},collab/.style={yshift=10pt},baseline=0]
\matrix[row sep=-0.4pt,column sep=-0.4pt] {
\node[cell,fill=gray!50](11){};&\node[cell,fill=gray!50](12){};&\node[cell,fill=gray!50](13){};&\node[cell,fill=gray!50](14){};&\node[cell](15){};&\node[cell](16){};&\node (17) [cell]{};\\
&\node[cell,fill=gray!50]{};&\node[cell,fill=gray!50]{};&\node[cell]{};&\node[cell]{};&\node[cell](26){};&\node[phcell](27){};\\
&&\node[cell,fill=gray!50]{};&\node[cell]{};&\node[cell](35){};&&\node[phcell](37){};\\
&&&\node[cell](44){};&&&\node[phcell](47){};\\};
\node[rowlab] at (17.east) {$1$};
\node[rowlab] at (27.east) {$2$};
\node[rowlab] at (37.east) {$3$};
\node[rowlab] at (47.east) {$4$};
\node[collab] at (17.north) {$2$};
\node[collab] at (16.north) {$3$};
\node[collab] at (15.north) {$4$};
\node[collab] at (14.north) {$0$};
\node[collab] at (13.north) {$-4$};
\node[collab] at (12.north) {$-3$};
\node[collab] at (11.north) {$-2$};
\end{tikzpicture}
\end{equation}
The set of roots corresponding to the gray cells in \eqref{PBN} is
$$
\{\e_i+\e_j\mid 1\le i<j\le 4\}\cup \{\e_1\}.
$$

\subsection{Proof of Theorem \ref{fb}}
To proceed in analogy with type $A$, we need to locate a  minimal inversion complete set of the supposed maximal cardinality, to have a lower bound
for $MC(B_n)$. Define
\begin{equation}\label{B}
P(B_n)=\{\e_i+\e_j\mid1\leq i<j\leq n\}\cup\{\e_1\}.\end{equation}
The gray cells in \eqref{PBN} afford  $P(B_4)$.

We now list all the possible decompositions of a positive root as a sum of two positive roots: if $i<j$
\begin{align}
\label{a}\e_i+\e_j&=(\e_i)+(\e_j)\\\label{b}&=(\e_i-\e_h)+(\e_h+\e_j),&1\leq i<h\leq n,\,h\ne j.
\\\label{bbb}&=(\e_i+\e_h)+(\e_j-\e_h),&1\leq j<h\leq n.\\
\label{c}
\e_i&=(\e_i-\e_h)+\e_h,& 1 \leq i<h\leq n.\\\label{dd}
\e_i-\e_j&=(\e_i-\e_h)+(\e_h-\e_j),&1\leq i<h<j\leq n.
\end{align}
If $\gamma=\e_i+\e_j$ with $i<j$ set
$$
L(\gamma)=\{\gamma\}\cup\{ \e_j\}
\cup\{\e_i-\e_h\mid i<h\leq n,\,h\ne j\}\cup\{\e_j-\e_h\mid 1\leq j<h\leq n\}.
$$
If $\gamma=\e_1$ set 
$$
L(\gamma)=\{\gamma\}\cup\{\e_1-\e_h\mid 1<h\leq n.\}.
$$
Notice that this is a slight modification of our definition of $L(\gamma)$ in type $A$. 
\begin{lemma} For any $\gamma\in P(B_n)$, $L(\gamma)$ is biclosed.
\end{lemma}
\begin{proof} Direct inspection of closedness and coclosedness, done using decompositions \eqref{a} to \eqref{dd}.\end{proof}
Let $w_\gamma\in W$ be such that $N(w_\gamma)=L(\gamma)$ and set $Y(B_n)=\{w_\gamma\mid \gamma\in P(B_n)\}$. 
Similarly to type $A$, one checks that 
\begin{prop}\label{Paolo} $Y(B_n)$ is a minimal inversion complete set. In particular, we have that  $MC(B_n)\geq\binom{n}{2}+1$.
\end{prop}

Let $Y$ be a minimal inversion complete set and let $S^+$ be an essential set for $Y$. By Proposition \ref{Paolo}, in order to prove Theorem \ref{fb}, we need only to to show that $|S^+|\le  \binom{n}{2}+1$. To check this we will use the graph associated to  $S=S^+\cup (-S^+)$. 

First of all, we see how the conditions of Proposition \ref{conditions} translate in our encoding of sets of roots as directed graphs. Let $\s$ be the arc reversing involution of $G(S)$ mapping each vertex $i$ into $-i$. 
 
 It is clear that to a root path $P\subset (\Dp)^k$ corresponds a pair of paths in the directed graph: so condition \eqref{nr1} of Proposition \ref{conditions} says that if two vertices $i,j$ with $i\ne \pm j$ are joined by two paths $P$, $P'$ in $\G(S)$ then either $P$ and $P'$ or $P$ and $\s(P')$ have an arc in common.
 
As in Remark \ref{no triangle},  condition \eqref{nr2} is equivalent to saying that in 
 $\G(S)$ there are no triangles.
 
 Condition \eqref{nr2} and condition \eqref{nr3} hold if and only if both in  $\G(S)$ and in  $\G_C(S)$ there are no triangles. Indeed, if a string of roots of length at least $2$ occurs, then there is a subgraph of $\G(S)$ of this type:
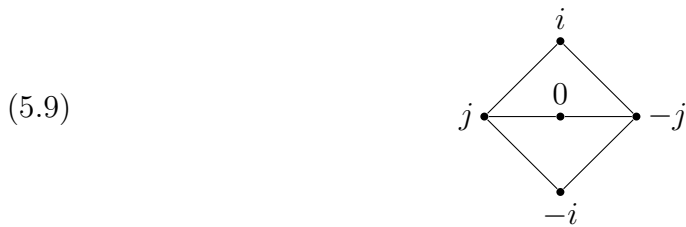
\begin{equation}\label{stringtwo}
\begin{tikzpicture}[baseline=0,dot/.style={circle,fill,inner sep=0,minimum size=3pt},left/.style={xshift=-5pt},up/.style={yshift=7pt},down/.style={yshift=-7pt},right/.style={xshift=10pt}]
\node[dot] (1) at (-1,0){};
\node[dot] (2) at (0,0){};
\node[dot] (3) at (1,0){};
\node[dot] (4) at (0,1){};
\node[dot] (5) at (0,-1){};
\draw[-] (1) -- (2);
\draw[-] (2) -- (3);
\draw[-] (3) -- (4);
\draw[-] (3) -- (5);
\draw[-] (1) -- (4);
\draw[-] (1) -- (5);
\node[up] at (4.north){$i$};
\node[up] at (2.north){$0$};
\node[left] at (1.west){$j$};
\node[down] at (5.south){$-i$};
\node[right] at (3.east){$-j$};
\end{tikzpicture}
\end{equation}
so in  $\G_C(S)$ there is a triangle. On the other hand, if there is a triangle in the graph of type $C$ and the triangle does not have edges $\{j,-j\}$ between opposite vertices, then there is a triangle in $\Gamma(S)$, contradicting condition \eqref{nr2}. If there is a triangle in $\Gamma_C(S)$ with an edge $\{j,-j\}$, then in $\Gamma(S)$ there is a subgraph as in \eqref{stringtwo}, so there is a string of roots of length at least $2$.

\subsubsection{Low rank cases}
By Example \ref{G2} we have that $MC(B_2)=2$. We now prove that $MC(B_n)=\binom{n}{2}+1$ for $n=3,4$. The proof will be a case by case check. 
\vskip0.5cm
\textsc{Case I: $B_3$.}\quad 
By Proposition \ref{Paolo}, we need only to check that $MC(B_3)\le 4$.
Let $Y$ be a minimal inversion complete set for $B_3$, $S^+$ an essential set for $Y$,  and  $S=S^+\cup -S^+$. We need to check that $|S|\le 8$. It clearly suffices to show that, if $S$ satisfies conditions \eqref{nr2} and \eqref{nr3} of Proposition \ref{conditions}, then $|S|\le 8$. 

 If no vertex is connected to $0$, then $\G(S)=\G_C(S)$ and $\G_C(S)$ cannot have triangles, so it has at most $9$ edges. Since, by construction,   $\G(S)$ has an even number of edges, we deduce that $|S|\le 8$. If there is a vertex $i$ connected to $0$, also $-i$ is connected to $0$. Consider the subgraph $\G'$ of $\G(S)$ induced by $V(S)\backslash\{\pm i\}$. The graph $\G'$ cannot have triangles and also its graph of type $C$ cannot have triangles, so, by what we proved in   the $B_2$ case, $\G'$ has at most $4$ edges.  If  $j\ne i$ and $j\ne 0$, then $j$  cannot be connected to both $i$ and $-i$, for, otherwise, there would be a triangle in the graph of type $C$. But then there is at most one edge connecting $j$ to $\pm i$. Since there are $4$ vertices different from $\pm i$ and different from $0$, there are at most $4$ edges connecting $\pm i$ to the other nonzero vertices. It follows that the total number of edges is less or equal to $10$. If $|S|=10$ then $i$ must be connected to at least $2$ vertices $j,h$ other than $0$ and $-i$ is connected  to $-i,-h$. Thus $\G(S)$ contains the graph
$$
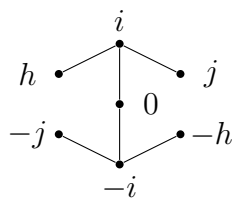
\begin{tikzpicture}[scale=0.2,
dot/.style={circle,fill,inner sep=0,minimum size=3pt},right/.style={xshift=10pt},up/.style={yshift=7pt},down/.style={yshift=-7pt},left/.style={xshift=-10pt}]
\node[dot] (0) at (0,0){};
\node[dot] (1) at (0,4){};
\node[dot] (2) at (4,2){};
\node[dot] (3) at (4,-2){};
\node[dot] (4) at (0,-4){};
\node[dot] (5) at (-4,-2){};
\node[dot] (6) at (-4,2){};
\node[right] at (0.east){$0$};
\node[up] at (1.north){$i$};
\node[right] at (2.east){$j$};
\node[right] at (3.east){$-h$};
\node[down] at (4.south){$-i$};
\node[left] at (5.west){$-j$};
\node[left] at (6.west){$h$};
\draw[-] (0) -- (1);
\draw[-] (0) -- (4);
\draw[-] (1) -- (2);
\draw[-] (1) -- (6);
\draw[-] (3) -- (4);
\draw[-] (4) -- (5);
\end{tikzpicture}
$$
It is now clear that we can only add $2$ edges to the graph above (the edge connecting $j$ and $-h$ and the edge connecting $-j$ and $h$) if we want to avoid triangles in $\G(S)$ and in $\G_C(S)$. This contradicts the assumption that $|S|=10$, and we are done in this case.
\vskip0.5cm
\textsc{Case II: $B_4$.}\quad In this case we will use also condition \eqref{nr1} of Proposition \ref{conditions}.

By Proposition \ref{Paolo}, we need only to check that $MC(B_4)\le 7$.
Let $Y$ be a minimal inversion complete set for $B_4$, $S^+$ an essential set for $Y$,  and  $S=S^+\cup -S^+$. We need to check that $|S|\le 14$. It clearly suffices to show that, if $S$ satisfies the conditions of Proposition \ref{conditions}, then $|S|\le 14$. 

Suppose first that there are no vertices connected to $0$. Then $\G(S)=\G_C(S)$. Since $\G_C(S)$ has no triangles, there are at most $16$ edges. Assume $|S|=16$.
In this case, by Tur{\'a}n's Theorem (\cite{Tu}), $\Gamma(S)$ is 
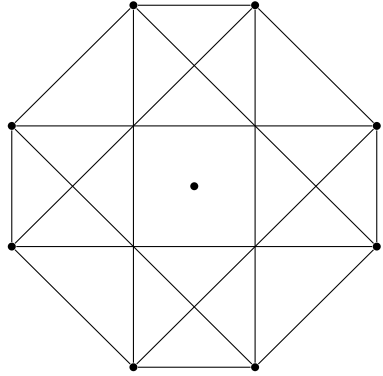
$$
\begin{tikzpicture}[scale=0.4,dot/.style={circle,fill,inner sep=0,minimum size=3pt},left/.style={xshift=10pt},up/.style={yshift=7pt},down/.style={yshift=-7pt},right/.style={xshift=-10pt}]
\node[dot] (0) at (0,0){};
\node[dot] (1) at (-2,6){};
\node[dot] (2) at (2,6){};
\node[dot] (3) at (6,2){};
\node[dot] (4) at (6,-2){};
\node[dot] (-1) at (2,-6){};
\node[dot] (-2) at (-2,-6){};
\node[dot] (-3) at (-6,-2){};
\node[dot] (-4) at (-6,2){};
\draw[-] (1) -- (2);
\draw[-] (1) -- (4);
\draw[-] (1) -- (-2);
\draw[-] (1) -- (-4);
\draw[-] (2) -- (3);
\draw[-] (2) -- (-1);
\draw[-] (-3) -- (2);
\draw[-] (3) -- (4);
\draw[-] (3) -- (-2);
\draw[-] (3) -- (-4);
\draw[-] (-1) -- (4);
\draw[-] (4) -- (-3);
\draw[-] (-1) -- (-2);
\draw[-] (-4) -- (-1);
\draw[-] (-2) -- (-3);
\draw[-] (-4) -- (-3);
\end{tikzpicture}
$$
In the graph of $S^+$, the vertex $1$ is a source and the vertex $-1$ is a sink. Note that two  vertices which are not connected can be switched without changing the graph. Thus, by rearranging the vertices we can assume that the graph of $S^+$ is, up to giving orientation to some edges,
$$
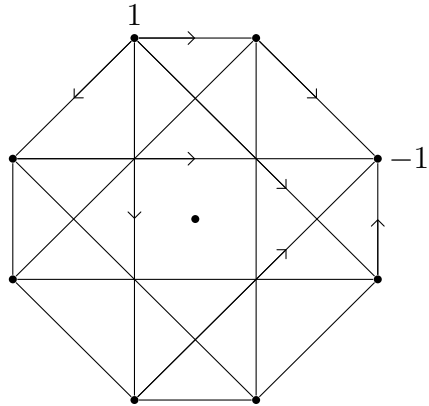
\begin{tikzpicture}[>=angle 90,scale=0.4,dot/.style={circle,fill,inner sep=0,minimum size=3pt},left/.style={xshift=10pt},up/.style={yshift=7pt},down/.style={yshift=-7pt},right/.style={xshift=-10pt}]
\node[dot] (0) at (0,0){};
\node[dot] (1) at (-2,6){};
\node[dot] (2) at (2,6){};
\node[dot] (3) at (6,2){};
\node[dot] (4) at (6,-2){};
\node[dot] (-1) at (2,-6){};
\node[dot] (-2) at (-2,-6){};
\node[dot] (-3) at (-6,-2){};
\node[dot] (-4) at (-6,2){};
\node[up] at (1.north){$1$};
\node[left] at (3.east){$-1$};
\draw[-] (1) -- (2);
\draw[->] (1) -- (0,6);
\draw[-] (1) -- (4);
\draw[->] (1) -- (3,1);
\draw[-] (1) -- (-2);
\draw[->] (1) -- (-2,0);
\draw[-] (1) -- (-4);
\draw[->] (1) -- (-4,4);
\draw[-] (2) -- (3);
\draw[->] (2) -- (4,4);
\draw[-] (2) -- (-1);
\draw[-] (-3) -- (2);
\draw[-] (3) -- (4);
\draw[->] (4) -- (6,0);
\draw[-] (3) -- (-2);
\draw[->] (-2) -- (3,-1);
\draw[-] (3) -- (-4);
\draw[->] (-4) -- (0,2);
\draw[-] (-1) -- (4);
\draw[-] (4) -- (-3);
\draw[-] (-1) -- (-2);
\draw[-] (-4) -- (-1);
\draw[-] (-2) -- (-3);
\draw[-] (-4) -- (-3);
\end{tikzpicture}
$$
If the path $1\to i\to -1$ occurs, also the path $1\to -i\to -1$ must occur. Up to rearranging the vertices we obtain the graph
$$
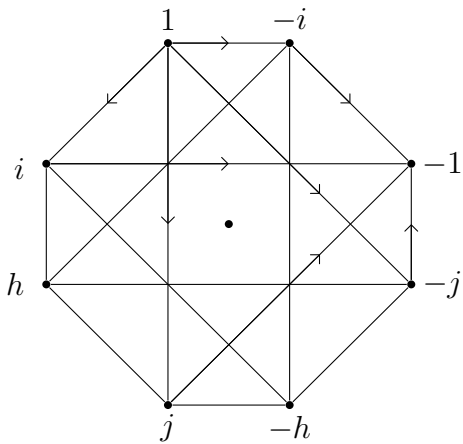
\begin{tikzpicture}[>=angle 90,scale=0.4,dot/.style={circle,fill,inner sep=0,minimum size=3pt},left/.style={xshift=10pt},up/.style={yshift=7pt},down/.style={yshift=-7pt},right/.style={xshift=-10pt}]
\node[dot] (0) at (0,0){};
\node[dot] (1) at (-2,6){};
\node[dot] (2) at (2,6){};
\node[dot] (3) at (6,2){};
\node[dot] (4) at (6,-2){};
\node[dot] (-1) at (2,-6){};
\node[dot] (-2) at (-2,-6){};
\node[dot] (-3) at (-6,-2){};
\node[dot] (-4) at (-6,2){};
\node[up] at (1.north){$1$};
\node[up] at (2.north){$-i$};
\node[left] at (3.east){$-1$};
\node[left] at (4.east){$-j$};
\node[down] at (-1.south){$-h$};
\node[down] at (-2.south){$j$};
\node[right] at (-3.west){$h$};
\node[right] at (-4.south){$i$};
\draw[-] (1) -- (2);
\draw[->] (1) -- (0,6);
\draw[-] (1) -- (4);
\draw[->] (1) -- (3,1);
\draw[-] (1) -- (-2);
\draw[->] (1) -- (-2,0);
\draw[-] (1) -- (-4);
\draw[->] (1) -- (-4,4);
\draw[-] (2) -- (3);
\draw[->] (2) -- (4,4);
\draw[-] (2) -- (-1);
\draw[-] (-3) -- (2);
\draw[-] (3) -- (4);
\draw[->] (4) -- (6,0);
\draw[-] (3) -- (-2);
\draw[->] (-2) -- (3,-1);
\draw[-] (3) -- (-4);
\draw[->] (-4) -- (0,2);
\draw[-] (-1) -- (4);
\draw[-] (4) -- (-3);
\draw[-] (-1) -- (-2);
\draw[-] (-4) -- (-1);
\draw[-] (-2) -- (-3);
\draw[-] (-4) -- (-3);
\end{tikzpicture}
$$
Let us choose the orientation for the edge $\{i,h\}$. If the orientation is $i\to h$, then we cannot have $h\to -j$ because of condition \eqref{nr1} (we have two paths $i\to h \to -j$ and $i\to -1 \to -j$ without common roots). We cannot have $-j\to h$ either because of condition \eqref{nr1} again (we have two paths $1\to i \to h$ and $1\to -j \to h$ without common roots). So we must have the orientation $h\to i$. But again, we cannot have $h\to -j$ because of condition \eqref{nr1} (we have two paths $h\to i \to -1$ and $h\to -j \to -1$ without common roots) and we cannot have $-j\to h$ either because of condition \eqref{nr1} (we have two paths $1\to -j \to h\to i$ and $1\to i$ without common roots). It follows that $|S|<16$ in this case. Since $|S|$ is even, we obtain that $|S|\le 14$.

If there is a pair of vertices $\pm i$ that is connected to $0$, consider the subgraph of $\Gamma(S)$ induced by $V(S)\backslash\{\pm i\}$. Then this graph satisfies the properties of Proposition \ref{conditions}, thus it can have at most $8$ edges. For each $j\ne\pm i$ such that $j\ne 0$ there is at most one edge connecting $i$ to $\pm j$  so the maximum number of such edges is $3$. The same argument holds for $-i$. Finally there are the two edges connecting $\pm i$ to $0$. The total is $8+3+3+2=16$, so $|S|\le 16$.

If the graph does have $16$ edges, then, for each $h\in\{\pm 1,\dots,\pm 4\}\backslash \{\pm i\}$, the vertex $i$ must be connected to exactly one between $h,-h$. 
If there are two pairs $\pm i$, $\pm j$ connected to $0$, then $i$ must be connected to either $j$ or $-j$ and this is impossible, for, otherwise, there would be a triangle in the graph of $S$. It follows that $\pm i$ are the only vertices connected to $0$.

Let $\pm j, \pm k,\pm h$ be the nonzero vertices other that $\pm i$. As shown above we can and do assume that $i$ is connected to $j,k,h$ and $-i$ is connected to $-j,-k,-h$.
There is no connection between the vertices $j,k,h$, for, otherwise, there would be a triangle. The same argument with $-i$ shows that there are no edges between $-j,-k,-h$. The only edges between $j,k,h,-j,-k,-h$ occur between vertices in $\{j,k,h\}$ and vertices in $\{-j,-k,-h\}$. The edges $\{r,-r\}$ between opposite vertices do not occur, so there are at most $6$ edges, short of the $8$ needed. It follows that $|S|<16$ and we are done.

\subsubsection{General case}
 
 \begin{theorem} 
$MC(B_n)=\binom{n}{2}+1$.
\end{theorem}
\begin{proof}
By Proposition \ref{Paolo}, we need only to prove that $MC(B_n)\le \binom{n}{2}+1$. This has been proved for $n=2,3,4$ above.
Note that we have actually proved in these cases  that, if $S^+$ satisfies the conditions of Proposition \ref{conditions}, then $|S^+|\le \binom{n}{2}+1$.

If $n\ge 4$, we will prove by induction on $n$ that, if a set $S^+\subset \Dp$ satisfies the conditions of Proposition \ref{conditions}, then $|S^+| \le \binom{n}{2}+1$. 
The base of the induction is the case $n=4$. Assume now $n>4$.

 Set $S=S^+\cup -S^+$. Assume first that there is a pair of vertices $\pm i$ connected to $0$.
Consider the subgraph of $\G(S)$ induced by   $V(S)\backslash\{\pm i\}$ as vertices. Then this graph satisfies the properties of Proposition \ref{conditions}, thus it can have at most $ 2\binom{n-1}{2}+2$ edges. For each $j\ne\pm i$ such that $j\ne 0$ there is at most one edge connecting $i$ to $\pm j$  so the maximum number of such edges is $n-1$. The same argument holds for $-i$. Finally there are the two edges connecting $\pm i$ to $0$. The total is $2\binom{n-1}{2}+2+2(n-1)+2=2\binom{n}{2}+4$. Thus $|S|\le 2\binom{n}{2}+4$. We need only to check that $|S|\ne 2\binom{n}{2}+4$.

If  $|S|=2\binom{n}{2}+4$, then, for each $h\in\{\pm 1,\dots,\pm n\}\backslash \{\pm i\}$, the vertex $i$ must be connected to exactly one between $h,-h$. 
If there is another pair  $\pm j$ connected to $0$, then $i$ must be connected to either $j$ or $-j$ and this is impossible, for, otherwise, there would be a triangle in the graph of $S$. 
It follows that at most one pair $\pm i$ is connected to $0$.
Let  $\pm i_1, \dots, \pm i_{n-1}$ be the other nonzero vertices. 
As observed above, $i$ must be connected to one element of each of the pairs $\pm i_1, \dots \pm i_{n-1}$, so assume $i$ is connected to $i_1,\dots,i_{n-1}$. Then there is no connection between the vertices $i_1,\dots,i_{n-1}$, for, otherwise, there would be a triangle. The same argument with $-i$ shows that there are no edges between $-i_1,\dots,-i_{n-1}$. The only edges between $\pm i_1, \dots \pm i_{n-1}$ occur between vertices in $\{i_1,\dots,i_{n-1}\}$ and vertices in $\{-i_1,\dots,-i_{n-1}\}$. The edges $\{r,-r\}$ between opposite vertices do not occur, so there are at most $2\binom{n-1}{2}$ edges, short of the $2\binom{n-1}{2}+2$ needed.
We deduce that 
the only possibility for having at least $2\binom{n}{2}+4$ edges can occur if no vertex is connected to $0$.

Assume now that $|S|\ge 2\binom{n}{2}+4$ and that there are no vertices connected to $0$. If we drop a vertex $\pm i$ from the graph, then, by the induction hypothesis,  we know that we can have at most $2\binom{n-1}{2}+2$ edges. Thus $\pm i$ have to be connected with at least $2n$ edges.  By symmetry, both $i$ and $-i$ are connected by at least $n$ edges, i.e. any vertex has degree at least $n$. This implies that each vertex has degree exactly $n$. In fact, if $i_1, \dots, i_{n+1}$ are connected to $i$, then $i_1$ can not be connected to $i_j$ for each $j$, because, otherwise, there would be triangles. So $i_1$ is connected only to at most $n-1$ edges. This is not possible because $i_1$ has degree at least $n$.  
Let $i_1,\dots,i_n$ be the vertices connected to $i$. As already observed they cannot be connected to each other, so they are connected to vertices in $\{\pm 1, \dots,\pm n\}\backslash \{i_1,\dots,i_n\}$. It follows that there are at least $n^2$ edges in the graph. If we drop $\pm i_1$ from the graph, we have at least $n^2-2n$ edges. But $
n^2-2n=2\binom{n-1}{2}+2+(n-4)$, so, since $n>4$, 
$$n^2-2n>2\binom{n-1}{2}+2
$$
and this is impossible due to the induction hypothesis.
\end{proof}
\section{Type $D$}
This case is treated along the lines of type $B$. 

\subsection{Encoding the roots.} Recall (cf. Planche IV of \cite{B}) that the root system can be described as
$$
\D=\{\pm(\e_i\pm\e_j)\mid 1\le i< j\le n\}.
$$
A set of positive roots is
$$
\Dp=\{\e_i\pm\e_j\mid 1\le i< j\le n\}.
$$

If $S$ is a set of roots, we define the graph $\G(S)$ by looking at $S$ as a subset of a root system of type $B$. Note that, since there are not short roots in $\D$, $\G(S)$ coincides with $\G_C(S)$ and in $\G_C(S)$ there are no edges $\{r,-r\}$ connecting opposite vertices.

\subsection{Proof of Theorem \ref{fd}}
Define
\begin{equation}\label{D}
P(D_n)=\{\e_i+\e_j\mid1\leq i<j\leq n\}.\end{equation}
If $i<j$ and $\g=\e_i+\e_j$, set 
$$L(\gamma)=\{ \e_i+\e_j\}
\cup\{\e_i-\e_h\mid i<h\leq n,\,h\ne j\}\cup\{\e_j-\e_h\mid 1\leq j<h\leq n\}.$$
One can check, as in Section \ref{Type B}, that $L(\g)$ is biclosed and that the set 
 $$Y(D_n)=\{w_\gamma\mid N(w_\gamma)=L(\gamma), \gamma \in P(D_n)\}
$$
is  minimal complete, of  cardinality $\binom{n}{2}$, so we have proved:
\begin{prop}\label{geDn} If $n\ge 4$, $MC(D_n)\geq\binom{n}{2}$.
\end{prop} 

\subsubsection{Case $D_4$} We want to prove that $MC(D_4)=6$. By Proposition \ref{geDn}, we need only to check that $MC(D_4)\le 6$.
Let $Y$ be a minimal inversion complete set for $D_4$, $S^+$ an essential set for $Y$,  and  $S=S^+\cup -S^+$. We need to check that $|S|\le 12$. It clearly suffices to show that, if $S$ satisfies the conditions of Proposition \ref{conditions}, then $|S|\le 12$. 

If we drop a pair $\pm i$ of vertices from $\G(S)$, then the resulting subgraph cannot have triangles, so there are at most $9$ edges. Since the number of roots must be even, the subgraph has at most $8$ edges. It follows that at least $6$ edges connect to $\pm i$. By the symmetry of the graph we see that $i$ has degree at least $3$.  Assume first that there is no vertex $r$ such that $i$ is connected to both $r$ and $-r$. Let $j,h,k$ be the vertices connected to $i$ so $-i$ is connected to $-j,-h,-k$. There cannot be any edge between the vertices  $j,h,k$, for, otherwise, there would be triangles. For the same reason there cannot be edges between the vertices $-j,-h,-k$. It follows that there are at most $6$ edges in the $\pm j, \pm h,\pm k$ subgraph. This implies that there are at most $12$ edges in $\G(S)$. 

We can therefore assume that $i$ is connected to a pair $\pm j$. By symmetry, $-i$ is connected to $\pm j$ as well. Since $i$ has degree at least $3$, there must be another vertex $h$ connected to $i$, so $-i$ is connected to $-h$. For the same reason $j$ is connected to another $k$. Note that $k\ne h$ for, otherwise, we would have triangles. Thus $\G(S)$ contains the subgraph
$$ 
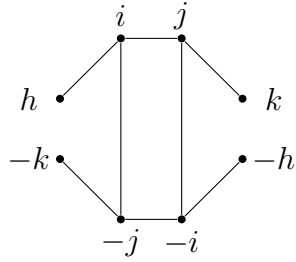
\begin{tikzpicture}[scale=0.2,dot/.style={circle,fill,inner sep=0,minimum size=3pt},left/.style={xshift=-10pt},up/.style={yshift=7pt},down/.style={yshift=-7pt},right/.style={xshift=10pt}]
\node[dot] (1) at (-2,6){};
\node[dot] (2) at (2,6){};
\node[dot] (3) at (6,2){};
\node[dot] (4) at (6,-2){};
\node[dot] (5) at (2,-6){};
\node[dot] (6) at (-2,-6){};
\node[dot] (7) at (-6,-2){};
\node[dot] (8) at (-6,2){};
\node[up] at (1.north){$i$};
\node[up] at (2.north){$j$};
\node[right] at (3.east){$k$};
\node[right] at (4.east){$-h$};
\node[down] at (5.south){$-i$};
\node[down] at (6.south){$-j$};
\node[left] at (7.west){$-k$};
\node[left] at (8.west){$h$};
\draw[-] (1) -- (2);
\draw[-] (2) -- (3);
\draw[-] (4) -- (5);
\draw[-] (5) -- (6);
\draw[-] (6) -- (7);
\draw[-] (1) -- (8);
\draw[-] (1) -- (6);
\draw[-] (2) -- (5);
\end{tikzpicture}
$$
The vertex $h$ must also have degree $3$ and there is a pair $\pm r$ connected to $h$. The possibilities for $r$ are $r=i$ or $r=k$ because both $j$ and $-j$ cannot be connected to $h$. The two resulting subgraphs are
$$ 
\begin{tikzpicture}[scale=0.2,dot/.style={circle,fill,inner sep=0,minimum size=3pt},left/.style={xshift=-10pt},up/.style={yshift=7pt},down/.style={yshift=-7pt},right/.style={xshift=10pt}]
\node[dot] (1) at (-2,6){};
\node[dot] (2) at (2,6){};
\node[dot] (3) at (6,2){};
\node[dot] (4) at (6,-2){};
\node[dot] (5) at (2,-6){};
\node[dot] (6) at (-2,-6){};
\node[dot] (7) at (-6,-2){};
\node[dot] (8) at (-6,2){};
\node[up] at (1.north){$i$};
\node[up] at (2.north){$j$};
\node[right] at (3.east){$k$};
\node[right] at (4.east){$-h$};
\node[down] at (5.south){$-i$};
\node[down] at (6.south){$-j$};
\node[left] at (7.west){$-k$};
\node[left] at (8.west){$h$};
\draw[-] (1) -- (2);
\draw[-] (2) -- (3);
\draw[-] (4) -- (5);
\draw[-] (5) -- node[yshift=-35pt]{$\G'$}(6);
\draw[-] (6) -- (7);
\draw[-] (1) -- (8);
\draw[-] (1) -- (6);
\draw[-] (2) -- (5);
\draw[-] (5) -- (8);
\draw[-] (4) -- (1);
\begin{scope}[xshift=25cm]
\node[dot] (1) at (-2,6){};
\node[dot] (2) at (2,6){};
\node[dot] (3) at (6,2){};
\node[dot] (4) at (6,-2){};
\node[dot] (5) at (2,-6){};
\node[dot] (6) at (-2,-6){};
\node[dot] (7) at (-6,-2){};
\node[dot] (8) at (-6,2){};
\node[up] at (1.north){$i$};
\node[up] at (2.north){$j$};
\node[right] at (3.east){$k$};
\node[right] at (4.east){$-h$};
\node[down] at (5.south){$-i$};
\node[down] at (6.south){$-j$};
\node[left] at (7.west){$-k$};
\node[left] at (8.west){$h$};
\draw[-] (1) -- (2);
\draw[-] (2) -- (3);
\draw[-] (4) -- (5);
\draw[-] (5) --node[yshift=-35pt]{$\G''$} (6);
\draw[-] (6) -- (7);
\draw[-] (1) -- (8);
\draw[-] (1) -- (6);
\draw[-] (3) -- (4);
\draw[-] (2) -- (5);
\draw[-] (7) -- (8);
\draw[-] (3) -- (8);
\draw[-] (7) -- (4);
  \end{scope}
\end{tikzpicture}
$$
In $\G'$ we have to add two edges that connect $k$ to a pair $\pm r$. The possibilities are $r=h$ or $r=i$, thus obtaining the two subgraphs
$$ 
\begin{tikzpicture}[scale=0.4,dot/.style={circle,fill,inner sep=0,minimum size=3pt},left/.style={xshift=-10pt},up/.style={yshift=7pt},down/.style={yshift=-7pt},right/.style={xshift=10pt}]
\node[dot] (1) at (-2,6){};
\node[dot] (2) at (2,6){};
\node[dot] (3) at (6,2){};
\node[dot] (4) at (6,-2){};
\node[dot] (5) at (2,-6){};
\node[dot] (6) at (-2,-6){};
\node[dot] (7) at (-6,-2){};
\node[dot] (8) at (-6,2){};
\node[up] at (1.north){$i$};
\node[up] at (2.north){$j$};
\node[right] at (3.east){$k$};
\node[right] at (4.east){$-h$};
\node[down] at (5.south){$-i$};
\node[down] at (6.south){$-j$};
\node[left] at (7.west){$-k$};
\node[left] at (8.west){$h$};
\draw[-] (1) -- (2);
\draw[-] (2) -- (3);
\draw[-] (4) -- (5);
\draw[-] (5) -- node[yshift=-35pt]{$\G'_1$}(6);
\draw[-] (6) -- (7);
\draw[-] (1) -- (8);
\draw[-] (1) -- (6);
\draw[-] (2) -- (5);
\draw[-] (5) -- (8);
\draw[-] (4) -- (1);
\draw[-] (7) -- (2);
\draw[-] (3) -- (6);
\begin{scope}[xshift=18cm]
\node[dot] (1) at (-2,6){};
\node[dot] (2) at (2,6){};
\node[dot] (3) at (6,2){};
\node[dot] (4) at (6,-2){};
\node[dot] (5) at (2,-6){};
\node[dot] (6) at (-2,-6){};
\node[dot] (7) at (-6,-2){};
\node[dot] (8) at (-6,2){};
\node[up] at (1.north){$i$};
\node[up] at (2.north){$j$};
\node[right] at (3.east){$k$};
\node[right] at (4.east){$-h$};
\node[down] at (5.south){$-i$};
\node[down] at (6.south){$-j$};
\node[left] at (7.west){$-k$};
\node[left] at (8.west){$h$};
\draw[-] (1) -- (2);
\draw[-] (2) -- (3);
\draw[-] (4) -- (5);
\draw[-] (5) --node[yshift=-35pt]{$\G'_2$} (6);
\draw[-] (6) -- (7);
\draw[-] (1) -- (8);
\draw[-] (1) -- (6);
\draw[-] (3) -- (4);
\draw[-] (2) -- (5);
\draw[-] (7) -- (8);
\draw[-] (3) -- (8);
\draw[-] (7) -- (4);
\draw[-] (1) -- (4);
\draw[-] (8) -- (5);
  \end{scope}
\end{tikzpicture}
$$
In $\G'_1$ an edge has to be added so that $k$ has degree at least $3$. The only possible choices are the edges $\{k,\pm h\}$. By possibly switching $h$ with $-h$, we can assume that $k$ is connected to $h$. The graph then becomes
$$ 
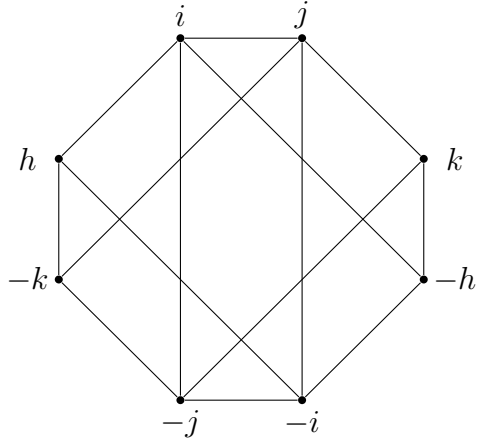
\begin{tikzpicture}[scale=0.4,dot/.style={circle,fill,inner sep=0,minimum size=3pt},left/.style={xshift=-10pt},up/.style={yshift=7pt},down/.style={yshift=-7pt},right/.style={xshift=10pt}]
\node[dot] (1) at (-2,6){};
\node[dot] (2) at (2,6){};
\node[dot] (3) at (6,2){};
\node[dot] (4) at (6,-2){};
\node[dot] (5) at (2,-6){};
\node[dot] (6) at (-2,-6){};
\node[dot] (7) at (-6,-2){};
\node[dot] (8) at (-6,2){};
\node[up] at (1.north){$i$};
\node[up] at (2.north){$j$};
\node[right] at (3.east){$k$};
\node[right] at (4.east){$-h$};
\node[down] at (5.south){$-i$};
\node[down] at (6.south){$-j$};
\node[left] at (7.west){$-k$};
\node[left] at (8.west){$h$};
\draw[-] (1) -- (2);
\draw[-] (2) -- (3);
\draw[-] (4) -- (5);
\draw[-] (5) -- (6);
\draw[-] (6) -- (7);
\draw[-] (1) -- (8);
\draw[-] (1) -- (6);
\draw[-] (2) -- (5);
\draw[-] (5) -- (8);
\draw[-] (4) -- (1);
\draw[-] (7) -- (2);
\draw[-] (3) -- (6);
\draw[-] (3) -- (4);
\draw[-] (7) -- (8);
\end{tikzpicture}
$$
which is, up to renaming the vertices, the subgraph $\G'_2$. The subgraph $\G''$ has $12$ edges. If $|S|>12$, there must be a vertex of degree $4$. Since in $\G''$ all vertices are equivalent, we can assume that $i$ has degree $4$. The only possibility is that $i$ is connected to $-h$ so also $-i$ is connected to $h$ and we obtain again $\G'_2$. Thus $\G'_2$ is a subgraph of $\G(S)$.  

We now give an orientation to the edges of the rectangle connecting the four vertices of degree $4$. The graph $\G(S)$ is symmetric with respect to the lines  parallel to the edges of this rectangle, so using this symmetry we can assume that the orientation in $G(S^+)$ for the edges $\{i,\pm h\}$ in $\G(S)$ is $i\to h$ and $i\to -h$. The digraph of $S^+$ thus contains
$$
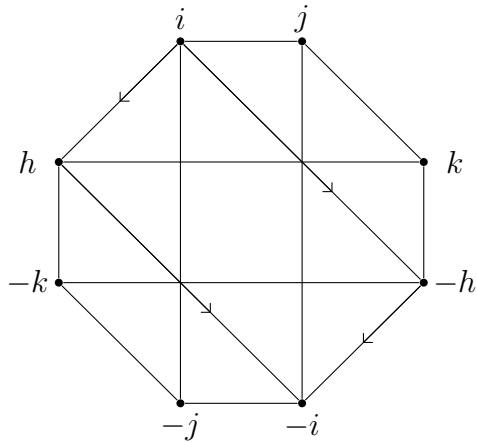
\begin{tikzpicture}[>=angle 90,scale=0.4,dot/.style={circle,fill,inner sep=0,minimum size=3pt},left/.style={xshift=-10pt},up/.style={yshift=7pt},down/.style={yshift=-7pt},right/.style={xshift=10pt}]
\node[dot] (1) at (-2,6){};
\node[dot] (2) at (2,6){};
\node[dot] (3) at (6,2){};
\node[dot] (4) at (6,-2){};
\node[dot] (5) at (2,-6){};
\node[dot] (6) at (-2,-6){};
\node[dot] (7) at (-6,-2){};
\node[dot] (8) at (-6,2){};
\node[up] at (1.north){$i$};
\node[up] at (2.north){$j$};
\node[right] at (3.east){$k$};
\node[right] at (4.east){$-h$};
\node[down] at (5.south){$-i$};
\node[down] at (6.south){$-j$};
\node[left] at (7.west){$-k$};
\node[left] at (8.west){$h$};
\draw[-] (1) -- (2);
\draw[-] (2) -- (3);
\draw[-] (4) -- (5);
\draw[->] (4) -- (4,-4);
\draw[-] (5) -- (6);
\draw[-] (6) -- (7);
\draw[-] (1) -- (8);
\draw[->] (1) -- (-4,4);
\draw[-] (1) -- (6);
\draw[-] (3) -- (4);
\draw[-] (2) -- (5);
\draw[-] (7) -- (8);
\draw[-] (3) -- (8);
\draw[-] (7) -- (4);
\draw[-] (1) -- (4);
\draw[->] (1) -- (3,1);
\draw[-] (8) -- (5);
\draw[->] (8) -- (-1,-3);
\end{tikzpicture}
$$
Let us choose the orientation for the edge $\{h,-k\}$. If the orientation is $h\to -k$, then we cannot have $-h\to -k$ because of condition \eqref{nr1} (we have two paths $i\to h \to -k$ and $i\to -h \to -k$ without common roots). We cannot have $-k\to -h$ either because of condition \eqref{nr1} again (we have two paths $i\to h \to -k\to -h$ and $i\to -h$ without common roots). So we must have the orientation $-k\to h$. But again, we cannot have $-k\to -h$ because of condition \eqref{nr1} (we have two paths $-k\to h\to -i$ and $-k\to -h \to -i$ without common roots) and we cannot have $-h\to -k$ either because of condition \eqref{nr1} (we have two paths $i\to -h \to -k\to h$ and $i\to h$ without common roots). It follows that $|S|<14$ and we are done.

\subsubsection{General case} 
 \begin{theorem} If  $n\ge 4$ then
$MC(D_n)=\binom{n}{2}$.
\end{theorem}
\begin{proof}
By Proposition \ref{geDn}, we need only to prove that $MC(D_n)\le \binom{n}{2}$. This has been proved for $n=4$ above.
Note that  we have actually proved  that, if $S^+$ satisfies the conditions of Proposition \ref{conditions}, then $|S^+|\le \binom{4}{2}$.

If $n\ge 4$, we will prove by induction on $n$ that, if a set $S^+\subset \Dp$ satisfies the conditions of Proposition \ref{conditions}, then $|S^+| \le \binom{n}{2}$. 
The base of the induction is the case $D_4$. Assume now $n>4$.

Set $S=S^+\cup -S^+$. Assume that $|S|\ge 2\binom{n}{2}+2$. If we drop a vertex $\pm i$ from the graph, then, by the induction hypothesis,  we know that we can have at most $2\binom{n-1}{2}$ edges. Thus $\pm i$ have to be connected with at least $2n$ edges.  By symmetry, both $i$ and $-i$ are connected by at least $n$ edges, i.e. any vertex has degree at least $n$. This implies that each vertex has degree exactly $n$. In fact, if $i_1, \dots, i_{n+1}$ are connected to $i$, then $i_1$ can not be connected to $i_j$ for each $j$, because, otherwise, there would be triangles. So $i_1$ is connected only to at most $n-1$ edges. This is not possible because $i_1$ has degree at least $n$.  

Let $i_1,\dots,i_n$ be the vertices connected to $i$. As already observed they cannot be connected to each other, so they are connected to vertices in
$$\{\pm 1,\dots,\pm n\}\backslash \{i_1,\dots,i_n\}.
$$
It follows that there are at least $n^2$ edges in the graph. If we drop $\pm i_1$ from the graph we have at least $n^2-2n$ edges. But $
n^2-2n=2\binom{n-1}{2}+(n-2)$, so, since $n>4$, 
$$n^2-2n>2\binom{n-1}{2}.
$$
and this is impossible due to the induction hypothesis.
\end{proof}
\section{Lower bounds for exceptional types}\label{exce}
We want to prove the inequalities \eqref{geq}. It suffices provide in each case a minimal complete set $Y$ of the required cardinality.
In  type $F_4$, if  we display the Dynkin diagram of $F_4$ as 
$
\begin{tikzpicture}[baseline=-3pt,double distance=2pt,>=stealth] 
\matrix [matrix of math nodes,row sep=0cm,column sep=1cm]
{|(0)|\circ&|(1)|\circ&|(2)|\circ&|(3)|\circ\\
};
\node [below] at (0) {$\a_1$};
\node [below] at (1) {$\a_2$};
\node (alfa) [below] at (2) {$\a_{3}$};
\node [below] at (3) {$\a_{4}$};
\begin{scope}[every node/.style={midway,auto}]
\draw (2) -- (3);
\draw (0) -- (1);
\draw[double,<-] (1) -- (2);
\end{scope}
\end{tikzpicture},
$
 then an essential set for $Y$ is given in the first column of the following table, while
$Y$ itself appears in the second column:
\begin{equation}\label{radF4}
\begin{tabular}{l|l}
$\a(w)$&$w$\\
\hline
$2\a_1+4\a_2+3\a_3+2\a_4$&$s_4 s_3 s_2 s_1 s_3 s_2 s_4 s_3$\\
$2\a_1+4\a_2+3\a_3+\a_4$&$s_3s_2s_1s_3s_2s_4s_3$\\
$2\a_1+4\a_2+2\a_3+\a_4$&$s_2s_1s_3s_2s_4s_3$\\
$2\a_1+3\a_2+2\a_3+\a_4$&$s_1s_2s_3s_2s_4s_3s_2s_1$\\
$\a_1+3\a_2+2\a_3+\a_4$&$s_2s_3s_2s_4s_3s_2s_1$\\
$\a_1+2\a_2+2\a_3+\a_4$&$s_3s_2s_4s_3s_2s_1$
\end{tabular}.
\end{equation}
\begin{rem}\label{f4}  In \cite{conti} we have performed calculations proving that $MC(F_4)=6$. Indeed,  we have proved that 
no subset of $\Dp$of cardinality bigger or equal to 7 can be  essential  for a minimal complete set. More precisely, we have first singled out the subsets
$\{\beta_1,\ldots,\beta_i\}\subset\Dp$ with $i\ge 7$ such that there exist $(w_1,\dots,w_i)\in W^i$ with
 $\beta_i\in N(w_k),\,\beta_j\notin N(w_k)$ for $j\ne k$. Then we have verified that no  choice of $(w_1,\dots,w_i)$ arising from the previous step affords a complete set.
\end{rem}
\vskip10pt
Now we have to deal with the $E$ series. Note that  $MC(T)$ is the maximal dimension of an abelian subalgebra of a simple Lie algebra of type $X=A,B,D$ (although  the same does not happen in types $C, F_4$, and $G_2$). Kostant's theory  implies that this number is also the maximal dimension of  an abelian ideal of a Borel subalgebra (see \cite{K} and also \cite{Pan}).
The previous remark   motivated the choice of $P(A_{n-1})$, $P(B_n)$, and $P(D_n)$: the sum of root subspaces corresponding to the roots displayed in  \eqref{A}, \eqref{B}, \eqref{D} is  indeed an abelian ideal of maximal dimension in each type. \par
Also for types $E$ we can exhibit a minimal complete set
choosing as essential set  $P(E_n),\,n=6,7,8$, the roots  indexing an abelian ideal of maximal dimension: we omit the details, referring to 
\cite{S} or \cite{CPA},  where these ideals are explicitly described.
To build up  $N(\a)$ for each $\a\in P(E_n)$, we observe that in type $E_6$ and $E_7$ for any decomposition $\a=\beta_i+\gamma_i$ as a sum of two positive roots, one and only one component, say $\beta_i$, belongs to $P(E_n)$. Then it turns out that $\{\a\}\cup\cup_i\{\gamma_i\}$ is biclosed, hence we can take it for $N(\a)$. More explicitly, in type $E_6$ we have
\begin{align*}Y(E_6)=\{ &s_2 s_4 s_3 s_5 s_4 s_2 s_6 s_5 s_4 s_3 s_1, s_4 s_3 s_5 s_4 s_2 s_6 s_5 s_4 s_3 s_1, s_3 s_5 s_4 s_2 s_6 s_5 s_4 s_3 s_1,\\&s_5 s_4 s_2 s_6 s_5 s_4 s_3 s_1, s_3 s_4 s_2 s_6 s_5 s_4 s_3 s_1, s_4 s_2 s_6 s_5 s_4 s_3 s_1, s_2 s_6 s_5 s_4 s_3 s_1,\\&s_3 s_4 s_2 s_5 s_4 s_3 s_1, 	s_4 s_2 s_5 s_4 s_3 s_1, s_2 s_5 s_4 s_3 s_1, s_2 s_4 s_3 s_1, s_6 s_5 s_4 s_3 s_1, s_5 s_4 s_3 s_1,\\& s_4 s_3 s_1, s_3 s_1, s_1\}\end{align*}
whereas in type $E_7$ we have
\begin{align*}Y(E_7)=\{ 
&s_1 s_3 s_4 s_2 s_5 s_4 s_3 s_1 s_6 s_5 s_4 s_2 s_3 s_4 s_5 s_6 s_7,
s_3 s_4 s_2 s_5 s_4 s_3 s_1 s_6 s_5 s_4 s_2 s_3 s_4 s_5 s_6 s_7,
\\&s_4 s_2 s_5 s_4 s_3 s_1 s_6 s_5 s_4 s_2 s_3 s_4 s_5 s_6 s_7,
s_2 s_5 s_4 s_3 s_1 s_6 s_5 s_4 s_2 s_3 s_4 s_5 s_6 s_7,
\\&s_2 s_4 s_3 s_1 s_6 s_5 s_4 s_2 s_3 s_4 s_5 s_6 s_7,
s_5 s_4 s_3 s_1 s_6 s_5 s_4 s_2 s_3 s_4 s_5 s_6 s_7,
\\&s_4 s_3 s_1 s_6 s_5 s_4 s_2 s_3 s_4 s_5 s_6 s_7,
s_3 s_1 s_6 s_5 s_4 s_2 s_3 s_4 s_5 s_6 s_7,
\\&s_1 s_6 s_5 s_4 s_2 s_3 s_4 s_5 s_6 s_7,
s_2 s_4 s_3 s_1 s_5 s_4 s_2 s_3 s_4 s_5 s_6 s_7,
\\&s_4 s_3 s_1 s_5 s_4 s_2 s_3 s_4 s_5 s_6 s_7,
s_3 s_1 s_5 s_4 s_2 s_3 s_4 s_5 s_6 s_7,
\\&s_1 s_5 s_4 s_2 s_3 s_4 s_5 s_6 s_7,
s_3 s_1 s_4 s_2 s_3 s_4 s_5 s_6 s_7,
\\&s_1 s_4 s_2 s_3 s_4 s_5 s_6 s_7,
s_1 s_2 s_3 s_4 s_5 s_6 s_7,
s_1 s_3 s_4 s_5 s_6 s_7,
s_6 s_5 s_4 s_2 s_3 s_4 s_5 s_6 s_7,
\\&s_5 s_4 s_2 s_3 s_4 s_5 s_6 s_7,
s_4 s_2 s_3 s_4 s_5 s_6 s_7,
s_2 s_3 s_4 s_5 s_6 s_7,
s_3 s_4 s_5 s_6 s_7,
\\&s_2 s_4 s_5 s_6 s_7,
s_4 s_5 s_6 s_7,
s_5 s_6 s_7,
s_6 s_7,
s_7
\}\end{align*}
In type $E_8$ it may happen that both components do not belong to $P(E_8)$, but it is again possible to choose the components in such a way to have a biclosed set. 
The outcome is as follows
{\scriptsize
\begin{align*}
Y(E_8)=\{ 
&s_8 s_7 s_6 s_5 s_4 s_2 s_3 s_1 s_4 s_3 s_5 s_4 s_2 s_6 s_5 s_4 s_3 s_7 s_6 s_5 s_4 s_2 s_8 s_7 s_6 s_5 s_4 s_3 s_1,
\\&s_7 s_6 s_5 s_4 s_2 s_3 s_1 s_4 s_3 s_5 s_4 s_2 s_6 s_5 s_4 s_3 s_7 s_6 s_5 s_4 s_2 s_8 s_7 s_6 s_5 s_4 s_3 s_1,
\\&s_6 s_5 s_4 s_2 s_3 s_1 s_4 s_3 s_5 s_4 s_2 s_6 s_5 s_4 s_3 s_7 s_6 s_5 s_4 s_2 s_8 s_7 s_6 s_5 s_4 s_3 s_1,
\\&s_5 s_4 s_2 s_3 s_1 s_4 s_3 s_5 s_4 s_2 s_6 s_5 s_4 s_3 s_7 s_6 s_5 s_4 s_2 s_8 s_7 s_6 s_5 s_4 s_3 s_1,
\\&s_4 s_2 s_3 s_1 s_4 s_3 s_5 s_4 s_2 s_6 s_5 s_4 s_3 s_7 s_6 s_5 s_4 s_2 s_8 s_7 s_6 s_5 s_4 s_3 s_1,
\\&s_2 s_3 s_1 s_4 s_3 s_5 s_4 s_2 s_6 s_5 s_4 s_3 s_7 s_6 s_5 s_4 s_2 s_8 s_7 s_6 s_5 s_4 s_3 s_1,
\\&s_1 s_2 s_4 s_3 s_5 s_4 s_2 s_6 s_5 s_4 s_3 s_7 s_6 s_5 s_4 s_2 s_8 s_7 s_6 s_5 s_4 s_3 s_1,
\\&s_3 s_1 s_4 s_3 s_5 s_4 s_2 s_6 s_5 s_4 s_3 s_1 s_7 s_6 s_5 s_4 s_2 s_3 s_8 s_7 s_6 s_5 s_4,
\\&s_1 s_4 s_3 s_5 s_4 s_2 s_6 s_5 s_4 s_3 s_1 s_7 s_6 s_5 s_4 s_2 s_3 s_8 s_7 s_6 s_5 s_4,
\\&s_1 s_3 s_5 s_4 s_2 s_6 s_5 s_4 s_3 s_1 s_7 s_6 s_5 s_4 s_2 s_3 s_8 s_7 s_6 s_5 s_4,
\\&s_1 s_3 s_4 s_2 s_6 s_5 s_4 s_3 s_1 s_7 s_6 s_5 s_4 s_3 s_8 s_7 s_6 s_5 s_4 s_2,
s_1 s_3 s_4 s_2 s_5 s_4 s_3 s_1 s_7 s_6 s_5 s_4 s_2 s_3 s_8 s_7 s_6 s_5 s_4,
\\&s_1 s_3 s_4 s_2 s_5 s_4 s_3 s_1 s_6 s_5 s_4 s_2 s_3 s_4 s_8 s_7 s_6 s_5,
s_2 s_4 s_3 s_5 s_4 s_2 s_6 s_5 s_4 s_3 s_7 s_6 s_5 s_4 s_2 s_8 s_7 s_6 s_5 s_4 s_3 s_1,
\\&s_4 s_3 s_5 s_4 s_2 s_6 s_5 s_4 s_3 s_1 s_7 s_6 s_5 s_4 s_2 s_8 s_7 s_6 s_5 s_4 s_3,
s_3 s_5 s_4 s_2 s_6 s_5 s_4 s_3 s_1 s_7 s_6 s_5 s_4 s_2 s_3 s_8 s_7 s_6 s_5 s_4,
\\&s_3 s_4 s_2 s_6 s_5 s_4 s_3 s_1 s_7 s_6 s_5 s_4 s_2 s_3 s_8 s_7 s_6 s_5 s_4,
s_3 s_4 s_2 s_5 s_4 s_3 s_1 s_7 s_6 s_5 s_4 s_2 s_3 s_8 s_7 s_6 s_5 s_4,
\\&s_3 s_4 s_2 s_5 s_4 s_3 s_1 s_6 s_5 s_4 s_2 s_3 s_8 s_7 s_6 s_5 s_4,
s_5 s_4 s_2 s_6 s_5 s_4 s_3 s_1 s_7 s_6 s_5 s_4 s_2 s_3 s_8 s_7 s_6 s_5 s_4,
\\&s_4 s_2 s_6 s_5 s_4 s_3 s_1 s_7 s_6 s_5 s_4 s_3 s_8 s_7 s_6 s_5 s_4 s_2,
s_4 s_2 s_5 s_4 s_3 s_1 s_7 s_6 s_5 s_4 s_2 s_3 s_8 s_7 s_6 s_5 s_4,
\\&s_4 s_2 s_5 s_4 s_3 s_1 s_6 s_5 s_4 s_2 s_3 s_4 s_8 s_7 s_6 s_5,
s_2 s_6 s_5 s_4 s_3 s_1 s_7 s_6 s_5 s_4 s_3 s_8 s_7 s_6 s_5 s_4 s_2,
\\&s_2 s_5 s_4 s_3 s_1 s_7 s_6 s_5 s_4 s_3 s_8 s_7 s_6 s_5 s_4 s_2,
s_2 s_5 s_4 s_3 s_1 s_6 s_5 s_4 s_2 s_3 s_8 s_7 s_6 s_5 s_4,
\\&s_2 s_4 s_3 s_1 s_7 s_6 s_5 s_4 s_2 s_3 s_8 s_7 s_6 s_5 s_4,
s_2 s_4 s_3 s_1 s_6 s_5 s_4 s_2 s_3 s_4 s_8 s_7 s_6 s_5,
\\&s_2 s_4 s_3 s_1 s_5 s_4 s_2 s_3 s_4 s_8 s_7 s_6 s_5,
s_1 s_3 s_4 s_2 s_5 s_4 s_3 s_1 s_6 s_5 s_4 s_2 s_3 s_7 s_6 s_5 s_4,
\\&s_3 s_4 s_2 s_5 s_4 s_3 s_1 s_6 s_5 s_4 s_2 s_3 s_7 s_6 s_5 s_4,
s_4 s_2 s_5 s_4 s_3 s_1 s_6 s_5 s_4 s_2 s_3 s_7 s_6 s_5 s_4,
\\&s_2 s_5 s_4 s_3 s_1 s_6 s_5 s_4 s_3 s_7 s_6 s_5 s_4 s_2,
s_2 s_4 s_3 s_1 s_6 s_5 s_4 s_2 s_3 s_7 s_6 s_5 s_4,
\\&s_2 s_4 s_3 s_1 s_5 s_4 s_2 s_3 s_4 s_7 s_6 s_5,
s_2 s_4 s_3 s_1 s_5 s_4 s_2 s_3 s_6 s_5 s_4
\}\end{align*}
}
These considerations prove the lower bounds 	\eqref{geq}.\par
In type $F_4$, the roots displayed in \eqref{radF4} index an abelian set (which is neither  of maximal dimension nor an upper ideal in the root poset). In type $C_n$, the maximal dimension of an abelian ideal is $\binom{n+1}{2}$, which is much larger than $MC(C_n)=MC(B_n)$. A conjectural conceptual explanation of this phenomenon is given in Section \ref{spec}.
\vskip20pt
\section{Noncrystallographic types}\label{nc}
We can formulate the  problem stated in Section \ref{3} for any finite reflection group. Indeed, let $W$ be a finite reflection group viewed as a finite Coxeter system with Coxeter generators $S$. Let $R=\bigcup\limits_{w\in W} wsw^{-1}$ be the set of reflections. Set
$$N(w)=\{t\in R\mid \ell(wt)<\ell(w)\}.$$
By \eqref{rinw}, the  previous formula is the natural generalization of the set of inversion to the noncrystallographic case. The notion of being inversion  complete
is changed in  the obvious way:  $Y\subset W$ is inversion  complete if $\bigcup\limits_{x\in Y}N(x)=R$.\par  Recall (see  \cite{D}, \cite[Chapter 5]{H}, \cite[4.2]{BB})  that any Coxeter group $(W,S)$ with relations $(ss')^{m_{s s'}}=1$ has a geometric representation obtained by considering a real vector space $V$ with a distinguished basis 
$\Pi=\{\a_s\mid s\in S\}$ on which $W$ operates by reflections w.r.t. the inner product $(\a_s,\a_{s'})=-\cos(\frac{\pi}{m_{s s'}})$ (which is positive definite if $W$ is finite). 
Set \begin{equation}\label{dr}\D_{Can}=W\Pi.\end{equation}
It is known that, for finite reflection groups, the following properties hold:
\begin{enumerate}
\item any $\a\in\D_{Can}$ can be written as $\a=\sum_{s\in S}a_s\a_s$ with either all $a_s\geq 0$ or all $a_s\leq 0$, but not in both ways;
correspondingly we decompose $\D_{Can}$ as $\D_{Can}=\Dp_{Can}\cup \D_{Can}^-$.
\item $s(\a_s)\in\D_{Can}^-,\, s(\Dp_{Can}\setminus\{\a_s\})\subset \Dp_{Can}.$
\item If $w\in W,\,s,s'\in S$ are such that $w(\a_s)=\a_{s'}$, then $wsw^{-1}=s'$.
\item If $\a\in\D_{Can}$ and, for some $t\in\R$, $t\a\in\D_{Can}$ then $t=\pm 1$.
\item If $\a\in\Dp_{Can}$ cannot  be written as $\a=x\beta+y\gamma,\,x>0,y>0,\beta,\gamma\in \Dp_{Can}$,  then $\a\in\Pi$.
\end{enumerate} 
These conditions ensure, in particular,  that inversions may be computed via \eqref{nw} also in this case.
\vskip10pt
We already proved in  Example \ref{G2} that $MC(I_2(m))=2$. 
In type $H_3$, if the Coxeter diagram is 
$
\begin{tikzpicture}[baseline=-3pt,double distance=2pt,>=stealth] 
\matrix [matrix of math nodes,row sep=0cm,column sep=1cm]
{|(0)|\circ&|(1)|\circ&|(2)|\circ\\
};
\node [below] at (0) {$\a_1$};
\node [below] at (1) {$\a_2$};
\node (alfa) [below] at (2) {$\a_{3}$};
\begin{scope}[every node/.style={midway,auto}]
\draw (0) to node[auto] {5} (1);
\draw (1) -- (2);
\end{scope}
\end{tikzpicture}
$ and $s_i\equiv s_{\a_i}$, the following subset is minimal complete
$$\{s_1 s_2 s_1 s_3 s_2 s_1, s_2 s_1 s_2 s_1 s_3 s_2 s_1, s_2 s_1 s_3 s_2 s_1, s_1 s_3 s_2 s_1, s_3 s_2 s_1\}.$$
The same procedure explained in Remark \ref{f4} for type $F_4$ can be implemented in this case to check that no essential set with at least  $7$ elements  exists and that all essential sets of cardinality $6$ is not the essential set of a minimal complete set.
(see  \cite{conti} for a {\sl Mathematica} code). \par
In type $H_4$  (with Coxeter diagram  $
\begin{tikzpicture}[baseline=-3pt,double distance=2pt,>=stealth] 
\matrix [matrix of math nodes,row sep=0cm,column sep=1cm]
{|(0)|\circ&|(1)|\circ&|(2)|\circ&|(3)|\circ\\
};
\node [below] at (0) {$\a_1$};
\node [below] at (1) {$\a_2$};
\node (alfa) [below] at (2) {$\a_{3}$};
\node (alfa) [below] at (3) {$\a_{4}$};
\begin{scope}[every node/.style={midway,auto}]
\draw (0) to node[auto] {5} (1);
\draw (1) -- (2);
\draw (2) -- (3);
\end{scope}
\end{tikzpicture}$
), we can provide a lower bound by exhibiting a minimal inversion complete set consisting
of 8 elements:
{\scriptsize
\begin{align}\label{yh4}&Y(H_4)=\\\notag &\{s_1s_3s_4s_3s_2s_1s_2s_1s_3s_2s_1s_2s_3s_4s_3s_2s_1s_2s_1s_3s_2s_1s_2s_3s_4s_3s_2s_1s_2s_1s_3s_2s_1s_2s_3s_4s_3s_2s_1s_2s_3,\\\notag
&s_2s_1s_2s_1s_4s_3s_2s_1s_2s_1s_3s_2s_1s_2s_3s_4s_3s_2s_1s_2s_1s_3s_2s_1s_2s_3s_4s_3s_2s_1s_2s_1s_3s_2s_1s_2s_3s_4,\\\notag
&
s_2s_1s_2s_3s_2s_1s_2s_1s_3s_4s_3,
s_2s_1s_3s_2s_1s_2s_1s_3s_2s_1s_4s_3,\\\notag
&
s_3s_2s_1s_2s_1s_4s_3s_2s_1s_2s_1s_3s_2s_1s_2s_3s_4s_3s_2s_1s_2s_1s_3s_2s_1s_2s_3s_4s_3s_2s_1s_2s_1s_3s_2s_1s_2s_3s_4,\\\notag
&
s_2s_3s_2s_1s_2s_3s_4s_3s_2s_1s_2s_1s_3s_2s_1s_2s_3s_4s_3s_2s_1s_2s_1s_3s_2s_1s_2s_3s_4s_3s_2s_1s_2s_1s_3s_2s_1s_2s_3s_4s_3s_2s_1s_2s_1s_3s_2s_1,\\\notag
&
s_2s_1s_2s_3s_2s_4s_3s_2s_1s_2s_1s_3s_2s_1s_2s_3s_4s_3s_2s_1s_2s_1s_3s_2s_1s_2s_3s_4s_3s_2s_1s_2s_1s_3s_2s_1s_2s_3s_4s_3s_2s_1s_2s_1s_3s_2s_1,\\\notag
&
s_3s_4s_3s_2s_1s_2s_1s_3s_2s_1s_2s_3s_4s_3s_2s_1s_2s_1s_3s_2s_1s_2s_3s_4s_3s_2s_1s_2s_1s_3s_2s_1s_2s_3s_4s_3s_2s_1s_2s_1s_3s_2s_1s_2\}
\end{align}}
\section{Final Comments}\label{spec}
In this section we  propose a conjectural approach to the calculation of $MC(T)$ which would explain the connection with the theory of abelian ideals of Borel subalgebras (see Section \ref{exce}).\par
Let $\D$ denote either a usual root system if  $W$ is crystallographic or $\D_{Can}$ (cf. \eqref{dr}) in the other cases.
\begin{defi}
We say  that   $A\subset \Dp$ is abelian if $\a,\beta\in A\implies \a+\beta\notin \D$.
We say that $A\subset \Dp$ is strongly abelian if $\a,\beta\in A\implies s\a+t\beta\notin \D\,\forall s,t\in\R_{>0}$.
\end{defi}
It is clear that  a strongly abelian set is abelian. The converse in general does not hold. We have however the following result.
\begin{prop}\label{sa} If $\D$ is of type $ADE$, then an abelian set is strongly abelian.
\end{prop}
\begin{proof} Let $A$ be an abelian set and consider $\a,\be\in A,\,\a\ne\be$. We may assume that all root have square length $2$.  Since we are in types
$ADE$, we have,  for any pair $\xi,\eta$ of distinct roots, that   $(\xi,\eta)\in\{-1,0,1\}$. Assume by contradiction that 
there exist positive scalars $x,y$ such that $\gamma=x\a+y\be\in\D$. Since $\a+\be\not\in \D$, we have that $(\a,\be)\ge 0$. If $(\a,\be)=0$ then $(\gamma,\a)=2x,(\gamma,\be)=2y$. The latter relations
force $x=y=\tfrac{1}{2}$. Then $||\gamma||^2=1$, which is absurd. If  $(\a,\be)=1$ then $(\gamma,\a)=2x+y,(\gamma,\be)=x+2y$ and in turn 
$x=y=\tfrac{1}{3}$. Then $||\gamma||^2=\tfrac{2}{3}$, which is again absurd.
\end{proof}  
 \vskip10pt
 \noindent{\bf Conjecture.\ }Ê There exists a minimal complete set $Y$ of maximal cardinality and a choice of an  essential set  for $Y$ which is  strongly abelian.
  \vskip5pt
 The conjecture is true in types $A_n$ and $D_n$, as  we have proven in the previous sections. It is also true in type $B_n$: the set $P(B_n)$ (cf. \eqref{B})  is indeed strongly abelian. It is easy to check directly the conjecture for $I_2(m)$. An essential set for the minimal complete set exhibited in Section \ref{nc} for $H_3$ is indeed strongly abelian. We also checked this is also a strongly abelian set of maximal cardinality (see \cite{conti}). It is also  easy to check directly that the maximal cardinality of a strongly abelian set of roots in type $F_4$ is $6$ (see \cite{conti}), which supports our belief that inequalities 
 in \eqref{geq} are indeed equalities: were the conjecture true, the theory of abelian
 subalgebras in semisimple Lie algebra (combined with Proposition \ref{sa}) would  yield a conceptual proof of  equality for types $E_n$ in \eqref{geq}.\par
 In type $H_4$, the  following statements have been checked in \cite{conti}:
 \begin{itemize}
 \item the maximal cardinality of  a strongly abelian set is  $10$;
 \item  no strongly abelian set of cardinality $10$ can be essential;
 \item all strongly abelian sets consisting of $8$ elements cannot be embedded in larger sets (even  not strongly abelian) that are the essential sets of minimal complete sets. 
 \end{itemize}
 We remark that  the unique essential set corresponding to \eqref{yh4}Ê is indeed strongly abelian.
 In particular, if the above conjecture is true, we have $MC(H_4)=8$.
 \par
\vskip5pt
\section{Acknowledgements}
We thank Nathan Reading for interesting comments on the paper.
Immediately after the appearance of this work in the ArXiv, Nathan Reading pointed out a numerical coincidence between our statistic $MC$ and the order dimension of   $W$ regarded as poset w.r.t. Bruhat order in types $A_n,B_n,I_2(n)$ (see \cite{R}). There can't be coincidence in type $F_4$: although the precise value for the order dimension is not known, the value of  $MC$ given in Remark \ref{f4} do not belong to the interval  given in \cite{R}  for the order  dimension.\par
Reading also noticed that   in types $A_n,B_n,I_2(n),H_3,H_4$ the Bruhat poset is dissective (and this property indeed allowed him to calculate the order dimension). This observation led us to compute $MC(H_3)$, which turns out be one less than the order dimension. 
\par
We also thank 
 Dmitri Panyushev for pointing out an inaccuracy in a previous version and the referee for his/her careful reading of the paper,  in particular for urging us to deal with the $H_4$ case.

\vskip5pt
\footnotesize{
\vskip5pt
{\sl Claudia Malvenuto, Luigi Orsina, Paolo Papi}: Dipartimento di Matematica, Sapienza Universit\`a di Roma, P.le A. Moro 2,
00185, Roma, Italy; 
\par\noindent
\par\noindent{\tt claudia@mat.uniroma1.it, orsina@mat.uniroma1.it,  papi@mat.uniroma1.it}
\vskip5pt
\noindent{\sl Pierluigi M\"oseneder Frajria}: Politecnico di Milano, Polo regionale di Como, 
Via Valleggio 11, 22100 Como,
ITALY;\\ {\tt pierluigi.moseneder@polimi.it}
}

\end{document}